\documentclass{amsart}
\usepackage[margin=1.76cm]{geometry}
\usepackage{amsthm}
\usepackage{amsmath}
\usepackage{amssymb}
\usepackage[all]{xy}
\usepackage{graphicx}
\usepackage{indentfirst}
\usepackage{bm}
\usepackage{mathrsfs}
\usepackage{latexsym}
\usepackage{color}
\usepackage{hyperref}
\usepackage{microtype}
\usepackage{tikz}
\usepackage{enumerate}
\usepackage{comment}
\usepackage{braket}
\usepackage{manfnt}
\usepackage{hyperref}
\usepackage[utf8]{inputenc}
\usepackage{relsize}
\usepackage{listings}
\usepackage{textcomp}
\usepackage{spverbatim}
\usepackage[T1]{fontenc}
\usepackage{lmodern}
\usepackage{url}
\usepackage{pdflscape}
\usepackage{enumitem}
\usepackage{appendix}

\DeclareMathOperator*{\lcm}{lcm}
\allowdisplaybreaks

\makeatletter
\@namedef{r@sec:modata}{{I}{1}{Catalogue of Integral Modular Data Up To Rank 13}{}{}}
\@namedef{r@NPMD}{{I.1}{1}{Non-Pointed Modular Data Up To Rank 13}{}{}}
\@namedef{r@sub:non-pointed}{{I.1}{1}{Non-Pointed Modular Data Up To Rank 13}{}{}}
\@namedef{r@subsub:R8F2}{{I.1.2}{2}{Rank $8$, type $[1,1,2,2,2,2,3,3]$, second fusion ring}{}{}}
\@namedef{r@subsubR11}{{I.1.5}{4}{Rank $11$, $[1,1,1,1,2,2,2,2,2,2,2]$}{}{}}
\@namedef{r@sub:pointed}{{I.2}{6}{Pointed Modular Data Up To Rank 13}{}{}}
\@namedef{r@sec:OddMD}{{II}{7}{Non-Pointed Odd-Dimensional Modular Data of Rank 17}{}{}}
\@namedef{r@subR17First}{{II.1}{7}{Type [[1,3],[3,8],[5,6]], First Fusion Ring}{}{}}
\@namedef{r@tocindent-1}{0pt}
\@namedef{r@tocindent0}{12.77776pt}
\@namedef{r@tocindent1}{20.27753pt}
\@namedef{r@tocindent2}{29.38896pt}
\@namedef{r@tocindent3}{0pt}
\makeatother

\begin{document}

% Theorems in italic (default)
\theoremstyle{plain}
\newtheorem{theorem}{Theorem}[section]
\newtheorem{main}{Main Theorem}
\newtheorem{proposition}[theorem]{Proposition}
\newtheorem{corollary}[theorem]{Corollary}
\newtheorem{lemma}[theorem]{Lemma}
\newtheorem{conjecture}[theorem]{Conjecture}
\newtheorem{claim}[theorem]{Claim}
\newtheorem{fact}[theorem]{Fact}
\newtheorem{question}[theorem]{Question}
\newtheorem*{st}{Statements}	% unnumbered

\numberwithin{equation}{section}

% Upright font for definitions, remarks, examples, etc.
\theoremstyle{definition}
\newtheorem{definition}[theorem]{Definition}
\newtheorem{notation}[theorem]{Notation}
\newtheorem{convention}[theorem]{Convention}
\newtheorem{example}[theorem]{Example}
\newtheorem{remark}[theorem]{Remark}
\newtheorem*{ac}{Acknowledgements}	% unnumbered

\newcommand{\Rep}{\mathrm{Rep}}
\newcommand{\Id}{\mathrm{id}}
\newcommand{\Aut}{\mathrm{Aut}}
\newcommand{\ch}{\mathrm{ch}}
\newcommand{\VVec}{\mathrm{Vec}}
\newcommand{\PSU}{\mathrm{PSU}}
\newcommand{\SO}{\mathrm{SO}}
\newcommand{\SU}{\mathrm{SU}}
\newcommand{\SL}{\mathrm{SL}}
\newcommand{\FPdim}{\mathrm{FPdim}}
\newcommand{\FPdims}{\mathrm{FPdims}}
\newcommand{\dims}{\mathrm{dims}}
\newcommand{\Tr}{\mathrm{Tr}}
\newcommand{\ord}{\mathrm{ord}}

\newcommand{\mC}{\mathcal{C}}
\newcommand{\mB}{B}
\newcommand{\hc}{\hom_{\mC}}
\newcommand{\id}{{\bf 1}} % the trivial object
\newcommand{\spec}{i_0} % point spectrum of PE
\newcommand{\Sp}{i_0} % point spectrum of TPE
\newcommand{\SpecS}{I_{s}} % spectrum set of PE
\newcommand{\SpS}{I_{s}'} % spectrum set of TPE
\newcommand{\field}{\mathbb{K}} % algebraically closed field
\newcommand{\white}{\textcolor{red}{\bullet}} %  i^** to i
\newcommand{\black}{\bullet} % i to i^**
\newcommand{\proofsketch}{ \noindent \textit{Proof sketch.} }
\newenvironment{restatetheorem}[1]
  {% begin code
   \par\addvspace{0.25\baselineskip}% Add vertical space similar to the theorem environment
   \noindent\textbf{Theorem \ref{#1}.}\ \itshape
  }
  {% end code
   \par\addvspace{0.25\baselineskip}
  }

\newcommand{\sebastien}[1]{\textcolor{blue}{#1 - Sebastien}}

\title{Classification of integral modular data up to rank 13}

\author{Max A. Alekseyev$^{a}$}
\address{M. A. Alekseyev}%, Department of Mathematics, George Washington University, Washington, DC, USA}
\email{maxal@gwu.edu}

\author{Winfried Bruns$^{b}$}
\address{W. Bruns}%, Institut für Mathematik, Universität Osnabrück, 49069 Osnabrück, Germany}
\email{wbruns@uos.de}

\author{Sebastien Palcoux$^{c}$}
\address{S. Palcoux}%, Beijing Institute of Mathematical Sciences and Applications, Huairou District, Beijing, China}
\email{sebastienpalcoux@gmail.com}
\urladdr{https://sites.google.com/view/sebastienpalcoux}

\author{Fedor V. Petrov$^{d}$}
\address{F.V. Petrov}%, St. Petersburg State University, St. Petersburg, Russia}
\email{f.v.petrov@spbu.ru}
\maketitle
%
%\begin{center}
%{\large
%Max A. Alekseyev}\\
%Department of Mathematics, George Washington University, Washington, DC 20052, USA\\
%maxal@gwu.edu
%
%\vspace{0.3cm}
%
%{\large
%Winfried Bruns}\\
%Institut für Mathematik, Universität Osnabrück, 49069 Osnabrück, Germany\\
%wbruns@uos.de
%
%\vspace{0.3cm}
%
%{\large
%Sebastien Palcoux}\\
%Beijing Institute of Mathematical Sciences and Applications (BIMSA), 
%Huairou District, Beijing 101408, China\\
%sebastienpalcoux@gmail.com
%
%\vspace{0.3cm}
%
%{\large
%Fedor V. Petrov}\\
%St. Petersburg State University, St. Petersburg 199034, Russia\\
%f.v.petrov@spbu.ru
%\end{center}
%
%
%\begin{center}
%\large
%Max A. Alekseyev$^{1}$, 
%Winfried Bruns$^{2}$, 
%Sebastien Palcoux$^{3}$, 
%Fedor V. Petrov$^{4}$
%\end{center}
%
\begin{center}
\footnotesize
$^{a}$Department of Mathematics, George Washington University, Washington, DC 20052, USA. 
%Email: maxal@gwu.edu

$^{b}$Institut für Mathematik, Universität Osnabrück, 49069 Osnabrück, Germany. 
%Email: wbruns@uos.de

$^{c}$Beijing Institute of Mathematical Sciences and Applications,% (BIMSA), 
Huairou District, Beijing 101408, China. 
%Email: sebastienpalcoux@gmail.com

$^{d}$St. Petersburg State University, St. Petersburg 199034, Russia. 
%Email: f.v.petrov@spbu.ru

\end{center}
\begin{abstract} 
This paper classifies the modular data of integral modular fusion categories up to rank 13, and integral half-Frobenius fusion rings up to rank 12. We establish that every perfect case within these bounds is trivial. Furthermore, we refine the non-pointed odd-dimensional modular data at ranks below 25 to exactly three items, all of rank 17, FPdim 225, and type [[1,3],[3,8],[5,6]], filling existing literature gaps. For rank 25, we narrow the perfect case to three open types.

Our core insight is that Egyptian fractions, typically used to list possible types, can be chosen with squared denominators. We develop several type criteria as initial filters. To construct the fusion rings, we solve dimension and associativity equations utilizing custom-built features in Normaliz. S-matrices are generated by self-transposing the character table, and T-matrices are derived by solving the Anderson-Moore-Vafa equations, concluding with the verification of extended modular data axioms.

From rank 13 onward, types are restricted by modular-specific properties involving universal grading, congruence representations of the modular group, and Galois action. This establishes critical arithmetic constraints: up to rank 21, a prime divisor of the global FPdim cannot exceed the rank, and up to rank 15 (non-pointed case), it cannot exceed half the rank. Ultimately, we reduce the rank 14 classification to 35 possible types, 8 of which are non-perfect.
\end{abstract}
%%In particular, we establish that the number of distinct basic FPdims in a non-trivial perfect fusion ring must be at least 4. 

\section{Introduction} 
In this paper, we assume that all fusion categories are defined over the complex field. The concept of an integral modular fusion category has been extensively studied, as detailed in the references at the beginning of \cite{CzPl}. In \cite{BrRo}, they have been classified up to rank $6$ (all pointed), with Egyptian fractions playing a crucial role. The approach that enables us to extend this classification up to rank $13$ in our work hinges on the observation that it is sufficient to consider Egyptian fractions with squared denominators. This restriction significantly reduces the combinatorial complexity. To illustrate this point, consider that the number of Egyptian fractions (summing to $1$) of length $n=1,2, \dots , 8$ is $1, 1, 3, 14, 147, 3462, 294314, 159330691$, respectively (as per \cite{A002966}). In contrast, when limited to squared denominators, the counts are $1, 0, 0, 1, 0, 1, 1, 4$, respectively (refer to \cite{A348625}).

We begin by recalling the concept of a fusion ring and its fundamental results in \S \ref{sub:Fu}, with reference to \cite[Chapter 3]{EGNO}. As defined in \cite{ENO21}, a fusion ring $\mathcal{F}$ is termed \emph{s-Frobenius} if for every basic element $b$, the ratio $\FPdim(\mathcal{F})^s / \FPdim(b)$ is an algebraic integer. According to \cite[Proposition 8.14.6]{EGNO}, the Grothendieck ring of a modular fusion category is $1/2$-Frobenius (denoted \emph{half-Frobenius} in the rest of the paper). Consider $\mathcal{F}$ to be an integral half-Frobenius fusion ring with a basis $\{b_1, \dots, b_r \}$,  $\FPdim \ D$, and type $[d_1, \dots, d_r]$, where $1=d_1 \le d_2 \le \cdots \le d_r$ and $d_i = \FPdim(b_i)$. Thus $d_i^2$ is a divisor of $D$, for all $i$. There exists a unique square-free integer $q$ such that $D = qs^2$, implying that each $d_i$ is a divisor of $s$. Let $s_i$ denote the positive integer $s/d_i$. Given that $D = \sum_{i=1}^r d_i^2$, we arrive at the following representation of $q$ as an Egyptian fraction with squared denominators $ q = \sum_{i=1}^r 1/s_i^2$. We have classified all such Egyptian fractions up to $r=13$ using SageMath, as will be discussed in \S \ref{sec:Egy}, where a method to constrain to $q \le r/4$ is also described (Proposition \ref{prop:reduction}). Since $s_1 = s$, we have $d_i = s_1 / s_i$, and we may assume that $s_i$ is a divisor of $s_1$, for all $i$. As detailed in \S \ref{sub:types}, this leads us to consider only $9025$ types up to rank $13$.

The subsequent phase entails implementing new criteria for identifying a type that emerges from a fusion ring, as delineated in \S \ref{sec:crit}.
%Particularly, we establish:
%\begin{restatetheorem}{thm:small}
%The minimum number of distinct basic $\FPdims$ in a non-trivial perfect fusion ring is four.
%\end{restatetheorem} 
The proof of these criteria predominantly relies on modular arithmetic and serves to rule out approximately $62\%$ of the types up to rank $13$. 

To address the remaining types, we classify all possible fusion data $(N_{i,j}^k)$, as defined in \S \ref{sub:Fu}, for each type $(d_i)$, utilizing our fusion ring solver described in \S \ref{sec:FRSolver}. We begin by reducing the number of variables, leveraging the Unit axiom of fusion data and the Frobenius reciprocity. The main challenge, denoted as \emph{patching}, involves integrating the associativity equations $\sum_s N_{i,j}^s N_{s,k}^t = \sum_s N_{j,k}^s N_{i,s}^t$ (which are non-linear) as efficiently as possible into the (linear) solving process of the dimension equations \(d_i d_j = \sum_k N_{i,j}^k d_k\), which are positive linear Diophantine equations. This approach was implemented using Normaliz \cite{Norma}, on which we developed new features dedicated to the classification of fusion rings, as explained in \S \ref{sub:user}, and for more details, see Appendix H of the manual \cite{NorManual}.

This step culminates in a classification of all the half-Frobenius integral fusion rings up to rank $12$, tallying exactly $10628$ instances derived from $71$ types, and proves the absence of any non-trivial perfect\footnote{Recall that \emph{perfect} means exactly one basic element with $\FPdim = 1$. A non-pointed simple fusion ring is perfect; \emph{pointed} means all basic elements have $\FPdim=1$, and \emph{simple} means no non-trivial proper fusion subrings.} integral half-Frobenius fusion rings up to rank $12$ (see \S \ref{sec:Half}). We can limit our attention to commutative fusion rings since a modular fusion category, being braided, possesses a commutative Grothendieck ring (although our classification encompasses $213$ noncommutative fusion rings as well; see \S \ref{sec:Half}).

From rank $13$ onward, the types were further restricted by additional properties coming from more advanced results on modular fusion categories, see \S \ref{sec:Egy} and \S \ref{sec:AdvMD}. This adjustment was necessary because we encountered computational limits for classifying all half-Frobenius fusion rings. Consequently, the result became less general at the fusion ring level compared to what we get up to rank $12$. In the non-perfect case, we applied specific \emph{universal grading} techniques, as discussed in \S \ref{sub:univ} and based on \cite[Proposition VI.2]{NRW23}. Additionally, we explored congruence representations of the modular group (see \S \ref{sub:cong}) and Galois actions (see \S \ref{sub:gal}). In particular, we provide two proofs of Theorem \ref{thm:folk}, a folklore\footnote{Throughout the paper, the term \emph{folklore} is used solely to indicate that the result is well-known to experts but not explicitly stated in the literature. It does not imply anything beyond that.} result: For any prime factor $p$ of the dimension norm of a modular fusion category with rank $r$, it holds that $p \leq 2r + 1$. The shorter proof applies \cite[Theorem II (iii)]{DLN}.
%Here, the \emph{norm} is defined as the product of the distinct Galois conjugates. 
This inequality is optimal, and the examples for which the equality holds are classified in \cite{NWZ}. Regarding the integral case, discussions with Eric Rowell and Andrew Schopieray indicated that the rank $r$ can be substituted with the multiplicity $m$ of a certain basic $\FPdim$, leading to the inequality $p \leq 2m+1$, see Theorem \ref{thm:RoSc}, required to exclude some hard types at rank $13$. We have a stronger version:  
\begin{restatetheorem}{thm:StrongPrime}
For an integral modular fusion category, let $S$ be the set of odd prime factors of the global $\FPdim$. There is a partition $(S_i)$ of $S$, and multiplicities $(m_i)$ of \emph{some} distinct basic $\FPdims$ such that $$m_i \ge \frac{1}{2} \lcm_{p \in S_i}(p-1).$$  
\end{restatetheorem} 
\noindent It was crucial for proving Proposition \ref{prop:MNSDper25} for the odd-dimensional case, and also the following theorem from \S \ref{sub:ArCo}:
\begin{theorem} \label{thm:strong}
Let $\mathcal{C}$ be an integral modular fusion category of rank $r$. For any prime factor $p$ of $\FPdim(\mathcal{C})$, the following bounds hold:
\begin{enumerate}
\item If $r \leq 21$ , then $p \leq r$.
\item If $r \leq 15$ and $\mathcal{C}$ is non-pointed, then $p \leq r/2$.
\end{enumerate}
In the non-perfect case, the above upper bounds can be improved to $r \leq 24$ and $r \leq 16$, respectively.
\end{theorem}
\noindent These inequalities are conjectured to hold without rank restriction in \S \ref{sub:ArCo}. These conjectures are proved for $\mathcal{Z}(\Rep(G))$, across all finite groups $G$, in \S \ref{sub:ZG}, among other results.
% Proposition \ref{prop:ConjR21} stating that for all ranks $r < 22$, then the inequality $p \leq r$ holds. This latter inequality is conjectured to hold in general (integral), and has been proven for $\mathcal{Z}(\Rep(G))$, across all finite groups $G$, see Proposition \ref{prop:conjcheck}, among other results in \S \ref{sub:ZG}.

At rank $13$, we obtained $28998$ fusion rings from $10$ types. The next step is to classify all possible modular data related to the fusion rings found up to rank $13$. The definition of modular data we employ (refer to \S \ref{sub:MD}) is informed by the key attributes of a modular fusion category, specifically a pseudo-unitary one, as our research is centered on the integral case (see \cite[Proposition 9.6.5]{EGNO}).

First, we examine the $S$-matrices: for a given commutative fusion ring, we take its eigentable (as defined in Definition \ref{def:eigentable}) and consider it as a matrix, retaining only those with cyclotomic elements—such fusion rings are termed \emph{cyclotomic}. If suitable renormalization and permutation yield a self-transpose matrix (detailed in \S \ref{sub:Smat}), we call the fusion ring as \emph{self-transposable}; if not, it is dismissed. From this, we infer that there are precisely $69$ self-transposable, cyclotomic, half-Frobenius, integral fusion rings up to rank $12$, originating from $27$ types, which is fewer than $0.7\%$ of the $10628$ identified in the initial stage. At rank $13$ (more restricted), we reduced to %$10+1$ such fusion rings from $2+1$ types. %did +1 for the pointed case, as in up to r12
$10$ non-pointed fusion rings from $2$ types.

Moving on to the $T$-matrices: for the fusion rings that remain, we solve the Anderson-Moore-Vafa equations (see \S \ref{sub:MD}) in the $\mathbb{Z}$-module $\mathbb{Q}/\mathbb{Z}$.
We preserve only those $S$- and $T$-matrices that satisfy all the conditions of Definition \ref{def:MD}. The use of our so-called \emph{magic criterion} was pivotal for several big cases and could lead to an interesting theoretical reformulation, see Question \ref{Q:magic}. Ultimately, we arrive at $19+64$ modular data, derived from $5+18$ fusion rings and $3+13$ types (non-pointed + pointed).

\begin{remark} \label{rem:pointedMD}
Every pointed modular fusion category corresponds to a metric group $(G,q)$—a finite Abelian group $G$ equipped with a non-degenerate quadratic form $q: G \to \mathbb{C}^*$, represented by the $T$-matrix, as described in \cite[\S 8.4]{EGNO}.
\end{remark}

The modular data (MD) mentioned in \cite[\S\ref{sec:modata}]{ABPPsupp} encompass $S$- and $T$-matrices, central charge, fusion data, and second Frobenius-Schur indicators for the non-pointed case. For the pointed case, however, it includes only the $T$-matrices. The following theorem, proved in \S \ref{sec:R13}, provides a concise overview:

\begin{theorem} \label{thm:main}
There are $19$ MD of non-pointed integral modular fusion categories up to rank $13$, given by: 
\begin{itemize}
\item Rank $8$, $\FPdim \ 36$, type $[1, 1, 2, 2, 2, 2, 3, 3]$: 
	\begin{itemize}
	\item $6$ MD with central charge $c=0$ from $\mathcal{Z}(\VVec_{S_3}^{\omega})$, see \cite{GrMo},
	\item $2$ MD with $c=4$ from $(C_3^2 + 0)^{C_2}$, see \cite[point (b) on page 983]{GeNaNi}. %the $C_2$-equivariantization of the Tambara-Yamagami category for $C_3^2$
	\end{itemize}
\item Rank $10$, $\FPdim \ 36$, type $[1, 1, 1, 2, 2, 2, 2, 2, 2, 3]$:
	\begin{itemize}
	\item  $3$ MD with $c=4$, from $SU(3)_3$, its complex conjugate and a zesting, see \cite[\S 6.3.1]{DeGaPlRoZh}.
	%Kac-Moody $A_2$ level $3$, see \cite{kac90},
	%the Verlinde modular category $\mathcal{C}(\mathfrak{su}(3),3)$
	\end{itemize}
\item Rank $11$, $\FPdim \ 32$, type $[1,1,1,1,2,2,2,2,2,2,2]$:
	\begin{itemize}
	\item $8$ MD with $c=\pm 1$, from $SO(8)_2$, a zesting and their Galois conjugates, see \S \ref{sub:mod11}. 
	% from Kac-Moody $D_4$ level $2$, see \cite{kac90},% from $SO(8)_2$ and variations, %from , conjugates and zestings, see \S \ref{sub:zes}.
	%the Verlinde modular category $\mathcal{C}(\mathfrak{so}(8),2)$
	\end{itemize}
\end{itemize}
There are $64$ modular data of pointed modular fusion categories up to rank $13$. Their number per group $G$ is in Table~\ref{tab:pointedMD}.
% for a cyclic group C_n the number of MD is given by https://oeis.org/A327730, as q_n(k) = exp(i pi k/n)
\end{theorem}
\begin{table}[h]
\[
\begin{array}{c|c|c|c|c|c|c|c|c|c|c|c|c|c|c|c|c|c|c}
G & C_1 & C_2 & C_3 & C_2^2 & C_4 & C_5 & C_6 & C_7 & C_2^3 & C_2 \times C_4 & C_8 & C_3^2 & C_9 & C_{10} & C_{11} & C_2 \times C_6 & C_{12} & C_{13}\\ \hline
\# \text{MD} & 1 & 2 & 2 & 5& 4 & 2 & 4 & 2 & 4 & 4 & 4 & 2 & 2 & 4 & 2 & 10 & 8 & 2
\end{array}
\]
\caption{Number of modular data for pointed modular fusion categories of rank at most $13$}
\label{tab:pointedMD}
\end{table}
We found no other integral modular data up to rank 13 apart from the categorifiable ones listed above.
\begin{question} \label{qu:MDcat}
Is there a modular data without categorification?
\end{question}

Note that \cite{NRW23} presents five intriguing non-integral candidates of ranks $11$ and $12$ (see \cite[I.C]{NRW23}), which are relevant to Question~\ref{qu:MDcat}; and it also agrees with\footnote{Our classification up to rank $12$ appeared in an earlier version as arXiv:2302.01613v3 in May 2023, three months before the release of \cite{NRW23} as arXiv:2308.09670 in August 2023. Our collaboration began in September 2021 at \url{https://mathoverflow.net/q/403397/34538}.} Theorem \ref{thm:main} up to rank $12$. This theorem yields the following consequence:

\begin{theorem} \label{thm:perfectcat13}
Every perfect integral modular fusion category up to rank $13$ is trivial. 
\end{theorem}  

In fact, we obtained the following more general result within the context of fusion rings up to rank $12$:

\begin{theorem} \label{thm:perfect12}
Every perfect integral half-Frobenius fusion ring up to rank $12$ is trivial.
\end{theorem}

The proof of Theorem \ref{thm:perfect12} for ranks up to $9$ is straightforward, following the list mentioned in \S \ref{sub:types} combined with an extended version of the Nichols-Richmond theorem applied to fusion rings, as detailed in the proof of \cite[Theorem 11]{NiRi}. This is due to the consistent presence of a non-trivial basic element with $\FPdim \le 2$. However, proving the theorem for ranks up to 12 necessitates the employment of type criteria, as discussed in \S \ref{sec:crit}, and the use of a fusion ring solver, elaborated in \S \ref{sec:FRSolver}.

It should be noted that the Drinfeld center of the representation category of any non-Abelian finite simple group $G$—and, more broadly, any centerless perfect group—is a perfect (though not simple) integral modular fusion category denoted as $\mathcal{Z}(\Rep(G))$ with $\FPdim = |G|^2$. For further information, see \cite[\S 11.1]{BuPa23}. Thus, the Grothendieck ring of $\mathcal{Z}(\Rep(A_5))$, of rank $22$ and type $[[1, 1], [3, 2], [4, 1], [5, 1], [12, 10], [15, 4], [20, 3]]$, constitutes a perfect integral half-Frobenius fusion ring. Consequently, Theorem \ref{thm:perfect12} cannot be extended to all ranks; however, it remains an open question whether its simple version can be:

\begin{question} \label{qu:SimpleIntHalfRing}
Is there a non-pointed simple integral half-Frobenius fusion ring?
\end{question}

A negative response to Question \ref{qu:SimpleIntHalfRing} would imply a negative answer to the renowned \cite[Question 2]{ENO11} in the simple case, due to a result in \cite{LPRinter}, which states that every simple integral fusion category is weakly group-theoretical if and only if every simple integral modular fusion category is pointed. With this in mind, we propose the following question:

\begin{question} \label{qu:SimpleIntModCat}
Is there a non-pointed simple integral modular fusion category?
\end{question}

For further insights into Question \ref{qu:SimpleIntModCat} at the fusion ring level, \cite[Corollary 6.16]{nik} adds a constraint: the absence of any prime-power basic $\FPdim$. It is worth noting that Theorem \ref{thm:perfect12} cannot be generalized to all ranks, even with this added constraint. This is because the Grothendieck ring of $\mathcal{Z}(\Rep(A_7))$, which is a perfect integral half-Frobenius fusion ring of rank $74$ and type $$ [[1, 1], [6, 1], [10, 2], [14, 2], [15, 1], [21, 1], [35, 1], [70, 9], [105, 4], [210, 20], [280, 9], [360, 14], [504, 5], [630, 4]],$$ satisfies this constraint (but consider Question \ref{Q:less74}). If necessary, Question \ref{qu:SimpleIntHalfRing} could be refined to include this constraint and the property of commutativity, and even more advanced constraints from \S \ref{sec:AdvMD}.
%his refinement almost allows for a negative resolution of Question \ref{qu:SimpleIntModCat} at rank $13$. In fact, there remain only two types to be examined, and nothing else \emph{in general} (see \S \ref{sec:R13}):

%\begin{theorem} \label{thm:R13}
%A non-pointed integral modular fusion category of rank $13$, if any, has one of the following types:
%\begin{enumerate}
%%\item [0.] $[1,1,1020,1292,1710,11628,14535,14535,19380,19380,19380,29070,29070]$,				%FPdim = 2^4 * 3^4 * 5^2 * 17^2 * 19^2
%\item[1.] $[1, 238, 459, 540, 595, 918, 5355, 9180, 21420, 21420, 32130, 32130, 32130]$, 		 %FPdim = 2^4 * 3^6 * 5^2 * 7^2 * 17^2
%\item[2.] $[1, 777, 1036, 1295, 3990, 4218, 24605, 42180, 98420, 98420, 147630, 147630, 147630]$. %FPdim = 2^4 * 3^2 * 5^2 * 7^2 * 19^2 * 37^2
%\end{enumerate}
%\end{theorem}

Using similar methods, in \S\ref{sec:r14r15} we reduce the rank-14 classification to just \(35\) possible types, \(27\) of which are perfect—each having an \(\FPdim\) whose set of prime divisors is exactly \(\{2,3,5,7\}\); see also Remark~\ref{ENOcrit}. Additionally, in conjunction with \cite[Remark 4.3]{CzPl} and \cite[Corollary 6.16]{nik}, we are able to demonstrate the following (refer to \S \ref{sec:MNSD}):
\begin{theorem} \label{thm:MNSD25}
There are three MD of non-pointed odd-dimensional modular fusion categories at rank below $25$: 
\begin{itemize}
\item Rank $17$, $\FPdim \ 225$, type $[1,1,1,3,3,3,3,3,3,3,3,5,5,5,5,5,5]$: 
	\begin{itemize}
	\item $3$ MD with central charge $c=4$ from $2$ fusion rings (see \cite[\S\ref{sec:OddMD}]{ABPPsupp} for the details).
	\end{itemize}
\end{itemize}
\end{theorem}
\begin{remark} \label{rk:IntroGaps} Theorem \ref{thm:MNSD25} does not align with \cite[Theorem 4.2, proof of Case (viii) $\FPdim(\mathcal{C}_{pt}) = p$]{BGHKNNPR} as well as \cite[Theorem 6.3 (b), proof of Case $|\mathcal{G}(\mathcal{C})|=3$]{CzPl}. But Remark \ref{rk:gaps} points out gaps in these proofs. Following our paper, \cite{CzPl} was corrected on arXiv, and \cite{GPR24} introduces modular category models for these new MD.
%Consider the candidates discussed in Remark \ref{rk:WithSeb}.
\end{remark}
Finally, this paper narrows down the possible rank $25$ odd-dimensional perfect types to $3$ ones, see Proposition \ref{prop:MNSDper25}.

\tableofcontents

\section{Fusion Data and Modular Data} \label{sec:FDMD}

In this section, we review the concepts of fusion data and modular data, along with the essential results. For further details, we refer the reader to \cite{EGNO}.

\subsection{Fusion Data} \label{sub:Fu}

The concept of fusion data expands upon the idea of a finite group.

\begin{definition} \label{def:fu}
\emph{Fusion data} consist of a finite set $\{1,2,...,r\}$ with an involution $i \mapsto i^*$, and nonnegative integers $N_{i,j}^k$ satisfying the following conditions for all $i,j,k,t$:
\begin{itemize}
\item (Associativity) $\sum_s N_{i,j}^s N_{s,k}^t = \sum_s N_{j,k}^s N_{i,s}^t$,
\item (Unit) $N_{1,i}^j = N_{i,1}^j = \delta_{i,j}$,
\item (Dual) $N_{i^*,j}^{1} = N_{j,i^*}^{1} = \delta_{i,j}$,
\item (Anti-involution) $N_{i,j}^{k} = N_{j^*,i^*}^{k^*}$.
\end{itemize}
Note that $1^* = 1$. We may represent the fusion data simply as $(N_{i,j}^k)$.
\end{definition}

\begin{proposition}[Frobenius Reciprocity] \label{prop:FrobRec}
For all $i,j,k$, $N_{i,j}^k = N_{k,j^*}^{i} = N_{k^*,i}^{j^*} =  N_{j^*,i^*}^{k^*} = N_{j,k^*}^{i^*} = N_{i^*,k}^j$.
\end{proposition}
\begin{proof}
Starting with (Associativity) and setting $t=1$, we have $\sum_s N_{i,j}^s N_{s,k}^1 = \sum_s N_{j,k}^s N_{i,s}^1$. Applying (Dual), we get $\sum_s N_{i,j}^s \delta_{s,k^*} = \sum_s N_{j,k}^s \delta_{s,i^*}$. Consequently, $N_{i,j}^{k^*} = N_{j,k}^{i^*}$. Substituting $k^*$ with $k$, we obtain $N_{i,j}^{k} = N_{j,k^*}^{i^*}$, which equals $N_{k,j^*}^{i}$ by (Anti-involution). The proposition follows by iterating the equality $N_{i,j}^k = N_{k,j^*}^{i}$ and (Anti-involution). 
\end{proof}

\begin{remark}
We can construct data that satisfy the first three axioms of Definition \ref{def:fu} but not the fourth, proving it is not superfluous. However, (Unit) is redundant when combined with the other axioms, as it is not utilized in the proof of Proposition \ref{prop:FrobRec}. Taken together, (Dual) and (Frobenius Reciprocity) trivially imply (Unit).
\end{remark}

A \emph{fusion ring} $\mathcal{R}$ is a free $\mathbb{Z}$-module equipped with a finite basis $\mathcal{B}=\{b_1, \dots, b_r\}$ and a fusion product defined by $$ b_i  b_j = \sum_k N_{i,j}^k b_k, $$ where $(N_{i,j}^k)$ constitutes fusion data, and a $*$-structure given by $b_i^* := b_{i^*}$. The four axioms for fusion data translate to the following for all $i,j,k$:
\begin{itemize}
\item  $(b_i  b_j)  b_k = b_i  (b_j  b_k)$, 
\item  $b_1  b_i = b_i  b_1 = b_i$,
\item  $\tau(b_i  b_j^*) = \delta_{i,j}$, 
\item  $(b_i  b_j)^* = b_j^*  b_i^*$,
\end{itemize}
where $\tau(x)$ is the coefficient of $b_1$ in the decomposition of $x \in \mathcal{R}$. Consequently, $\mathcal{R}_{\mathbb{C}} := \mathcal{R} \otimes_{\mathbb{Z}} \mathbb{C}$ becomes a finite-dimensional unital $*$-algebra, with $\tau$ extending linearly to a trace (i.e., $\tau(xy) = \tau(yx)$) and an inner product defined by $\langle x,y \rangle := \tau(x y^*)$. Here, $\langle x,b_i \rangle$ is the coefficient of $b_i$ in the decomposition of $x$.

\begin{theorem}[Frobenius-Perron Dimension Theorem {\cite[\S 8]{ENO05}}] \label{thm:FrobPer}
Given a fusion ring $\mathcal{R}$ with basis $\mathcal{B}$ and the corresponding finite-dimensional unital $*$-algebra $\mathcal{R}_{\mathbb{C}}$, there exists a unique $*$-homomorphism $d:\mathcal{R}_{\mathbb{C}} \to \mathbb{C}$ such that $d(\mathcal{B}) \subset \mathbb{R}_{>0}$.
\end{theorem}

The value $d(b_i)$, known as the \emph{Frobenius-Perron dimension} of $b_i$, is denoted as $\FPdim(b_i)$ or simply $d_i$. This is referred to as a \emph{basic $\FPdim$}. The sum $\sum_i d_i^2$ is called the Frobenius-Perron dimension of $\mathcal{R}$, or \emph{global $\FPdim$}, and is denoted $\FPdim(\mathcal{R})$. The sequence $[d_1, d_2, \dots, d_r]$ is called the \emph{type} of $\mathcal{R}$. A fusion ring $\mathcal{R}$ is described as:
\begin{itemize}
\item \emph{Frobenius} (or $1$-Frobenius, or of Frobenius type) if $\frac{\FPdim(\mathcal{R})}{\FPdim(b_i)}$ is an algebraic integer for all $i$,
\item \emph{integral} if $\FPdim(b_i)$ is an integer for all $i$,
\item \emph{pointed} if $\FPdim(b_i)=1$ for all $i$,
\item \emph{commutative} if $b_i  b_j = b_j  b_i$ for all $i,j$, meaning $N_{i,j}^k = N_{j,i}^k$.
\end{itemize}

The \emph{multiplicity} of $\mathcal{R}$ is the maximum value among $N_{i,j}^k$, and its \emph{rank} is $r$, the size of the basis.

\begin{remark} \label{rem:mat}
Fusion data enable a representation of its corresponding fusion ring. Consider the matrices $M_i = (N_{i,j}^k)_{k,j}$. By (Associativity) in Definition \ref{def:fu}, we verify that $M_i M_j = \sum_k N_{i,j}^k M_k$. Additionally, $M_1$ is the identity matrix, and Frobenius Reciprocity ensures that the adjoint matrix $M_i^*$ is $M_{i^*}$. According to Frobenius-Perron Theorem, the operator norm $\|M_i\|$ equals $\FPdim(b_i)$.
\end{remark}

\begin{remark}
The concept of fusion data is a combinatorial reformulation of the fusion ring notion, so any property applicable to a fusion ring is also applicable to its fusion data.
\end{remark}

\begin{definition}[Eigentable] \label{def:eigentable}
Given commutative fusion data $(N_{i,j}^k)$, consider the corresponding fusion matrices $M_i = (N_{i,j}^k)_{k,j}$. The commutativity and the property that $M_i^* = M_{i^*}$ render these matrices normal and thus simultaneously diagonalizable. Let $(D_i)$ denote their simultaneous diagonalization, where $D_i = \mathrm{diag}(\lambda_{i,j})$. We can select $\lambda_{i,1}= \| M_i \| =d_i$. The matrix $(\lambda_{i,j})$ is termed the \emph{eigentable} (or character table) of the fusion data, and the values $c_j:= \sum_i |\lambda_{i,j}|^2$ are called the \emph{formal codegrees}.
\end{definition}

\begin{lemma} \label{lem:dim1}
Let $M \in M_n(\mathbb{Z}_{\ge 0})$. The matrix $M$ is a permutation matrix if and only if $\| M \| = 1$.
\end{lemma}
\begin{proof}
Consider an orthonormal basis $\{e_1, \dots, e_n\}$ for which the entries of $M$ are non-negative integers. If $M$ is not a permutation matrix, then one of the following cases must occur:

\begin{itemize}
\item[(0)] there exists $i$ for which $Me_i = 0$,
\item[(1)] there exist $i,j$ such that $\langle Me_i , e_j \rangle  > 1$,
\item[(2)] there exist $i,j,k$ with $j \neq k$, such that $\langle Me_i , e_j \rangle = \langle Me_i , e_k \rangle = 1$,
\item[(3)] there exist $i,j,k$ with $i \neq j$, such that $Me_i = Me_j = e_k$.
\end{itemize}

However, case (0) implies $\| Me_i \|/\| e_i \|  = 0$, while case (1) leads to $\| Me_i \|/\| e_i \| > 1$. In case (2), it follows that $\| Me_i \|/\| e_i \| \ge \sqrt{2}$. Likewise, case (3) implies $\| M(e_i+e_j) \|/\| e_i+e_j \| = \sqrt{2}$. Each of these cases indicates that $\| M \| > 1$. Conversely, if $M$ is a permutation matrix, it trivially follows that $\| M \| = 1$.
\end{proof}

\begin{corollary} \label{cor:dim1}
For two basic elements $x,y$ of a fusion ring with $\FPdim(x) = 1$, both $x y$ and $y x$ are basic elements, and $\FPdim(x y) = \FPdim(y x) = \FPdim(y)$.
\end{corollary}
\begin{proof}
This follows directly from Remark \ref{rem:mat}, Lemma \ref{lem:dim1}, and the fact that $\FPdim$ is a ring homomorphism.
\end{proof}

\begin{corollary} \label{cor:pointed}
A fusion ring is pointed if and only if its basis forms a finite group under the fusion product.  
\end{corollary}

\subsection{Modular Data}
\label{sub:MD}
Broadly speaking, modular data refers to a fusion data together with two matrices, $S$ and $T = (t_{i,j})$, that generate a projective representation of the modular group $\SL(2,\mathbb{Z})$. To provide a more detailed description, we draw upon \cite[Theorem 2.1]{NRWW} and \cite[\S 8.13, \S 8.18]{EGNO}. Let $\mathbf{i}$  be the imaginary unit.

\begin{definition}
\label{def:MD}
Given a fusion ring $\mathcal{R}$ of rank $r$, type $[d_1, \dots, d_r]$, and fusion data $(N_{i,j}^k)$, let $\mathsf{d}:= \FPdim(\mathcal{R})$ and $\zeta_n := \exp(2 \pi \mathbf{i} /n)$. A (pseudounitary) \emph{modular data} for $\mathcal{R}$ consists of two matrices $S, T \in M_r(\mathbb{C})$ satisfying:

\begin{itemize}
\item $S$ and $T$ are symmetric, $T$ is unitary and diagonal with $T_{1,1}=1$, $S_{1,i}=d_i$ for all $i$, and $S S^* = \mathsf{d} \id$.
\item Verlinde formula:  $N_{i,j}^k = \frac{1}{\mathsf{d}}\sum_l \frac{S_{li} S_{lj} \overline{S_{lk}}}{d_l}$.
\item Twist: let $\theta_i$ be $T_{i,i}$, then $\sum_k N_{i,j}^k \theta_k d_k = \theta_i \theta_j S_{i,j}$.
\item Ribbon structure: $\theta_i=\theta_{i^*}$ (see Remark \ref{rem:rib}).
\item Central charge: $p_\pm := \sum_{i=1}^{r} d_i^2 (\theta_i)^{\pm 1}$. The ratio $p_+/p_-$ is a root of unity, and $p_+=\sqrt{\mathsf{d}}\zeta_8^c$ for some rational number $c$, referred to as the \textbf{central charge}, determined modulo $8$.
\item The matrices $S$ and $T$ afford a projective representation of $\SL(2,\mathbb{Z})$: we have $(ST)^3 = p_+ S^2$, $\frac{S^2}{\mathsf{d}} = C$, $C^2 = \id$, where $C$ is the permutation matrix associated with the involution $i \to i^*$ and satisfies $\Tr(C)> 0$.
\item Cauchy theorem: the set of distinct prime factors of $\ord(T)$ is identical to the distinct prime factors of $\mathrm{norm}(\mathsf{d})$, where $\mathrm{norm}(x)$ denotes the product of the distinct Galois conjugates of the algebraic number $x$.
\item Cyclotomic integers: for all $i, j$, the elements $S_{i,j}$, $S_{i,j}/d_j$ and $T_{i,i}$ are cyclotomic integers. The conductor of $S_{i,j}$ divides $\ord(T)$, which in turn divides $\mathsf{d}^{5/2}$, and there exists $j$ such that $S_{i,j}/d_j \in \mathbb{R}_{\ge 1}$, for all $i$.
\item Frobenius-Schur indicators: for every $i$ and for all $n \ge 1$, the sum $\nu_n(i):= \frac{1}{\mathsf{d}} \sum_{j,k} N_{j,k}^i (d_j\theta_j^n) \overline{(d_k\theta_k^n)}$ is a cyclotomic integer with a conductor that divides both $n$ and $\ord(T)$. Additionally, $\nu_1(i)=\delta_{i,1}$ and $\nu_2(i)=\pm \delta_{i,i^*}$.
\item Anderson-Moore-Vafa equations: $T_{i,i} = e^{2 \pi \mathbf{i} t_i}$, and $\forall i, j, k, l$, the following equation holds in the $\mathbb{Z}$-module $\mathbb{Q}/\mathbb{Z}$:
$$\left(\sum_{p=1}^r N_{i,j}^p N_{p,k}^l \right) (t_i + t_j + t_k + t_l) = \sum_{p=1}^r \left(N_{i,j}^p N_{p,k}^l + N_{i,k}^p N_{j,p}^l + N_{j,k}^p N_{i,p}^l \right) t_p.$$
The \textbf{topological spin} of the $i$-th basic element is the representative $s_i \in (-1/2,1/2]$ of $t_i \in \mathbb{Q}/\mathbb{Z}$. 
\end{itemize}
\end{definition}

We could question the necessity of each component in Definition \ref{def:MD}, particularly whether the Anderson-Moore-Vafa equations can be inferred from the other assumptions.

\begin{remark}
\label{rem:CoHalf}
The Verlinde formula, in conjunction with results from \cite[\S 2]{LPRinter}, implies that the fusion ring $\mathcal{R}$ is commutative. Together with $S$ symmetric and the identity $SS^* = \FPdim(\mathcal{R})\id$, it can be deduced that $\mathcal{R}$ is self-transposable (as discussed in \S \ref{sub:Smat}). Moreover, according to the proof presented in \cite[Proposition 8.14.6]{EGNO}, $\mathcal{R}$ is also half-Frobenius.
\end{remark}

\begin{remark}
\label{rem:rib}
A modular tensor category $\mathcal{C}$ possesses a ribbon structure, which means that the twist $\theta \in \Aut(\Id_{\mathcal{C}})$ satisfies the condition $(\theta_X)^* = \theta_{X^*}$ for every object $X$ within $\mathcal{C}$. Let $(X_i)$ represent the set of simple objects (up to isomorphism) within $\mathcal{C}$. Schur's lemma guarantees that $\theta_{X_i} = \theta_i \Id_{X_i}$, where the scalar $\theta_i$ is consistent with the one described in Definition \ref{def:MD}. Owing to the ribbon structure, we deduce the following:
\[
\theta_{i^*} \Id_{X_{i^*}} = \theta_{X_{i^*}} = (\theta_{X_i})^* = (\theta_i \Id_{X_i})^* = \theta_i (\Id_{X_i})^* = \theta_{i} \Id_{X_{i^*}}.
\]
From this, it follows that $\theta_{i^*} = \theta_{i}$ for all simple objects $X_i$.
\end{remark}

This paper primarily concerns integral fusion categories. In this setting, $\mathsf{d}$ is an integer, so $\mathrm{norm}(\mathsf{d}) = \mathsf{d}$. These categories are pseudounitary and therefore spherical (see \cite[Proposition 9.5.1]{EGNO}). In contexts that are not pseudounitary, Definition \ref{def:MD} would require modifications (as suggested in \cite[Theorem 2.1]{NRWW}) because the equalities $S_{1,i} = \FPdim(b_i)$, for all $i$, may not be valid. Strictly speaking, the assumption arising from these last equalities should be referred to as \emph{positive}, which lies between \emph{pseudounitary} and \emph{unitary}. Recall that \cite[Proposition 9.5.1]{EGNO} states that a pseudounitary fusion category admits a positive structure, but whether it always admits a unitary structure remains an open problem (see \cite[page 284]{EGNO}).

It should be noted that the definition of modular data provided here is so stringent that, as of now, no instances exist that lack a categorification, leading to Question \ref{qu:MDcat}.

\section{From Fusion Data to Modular Data}
\label{sec:FuToMD}
This section elucidates the classification of all potential modular data associated with a given set of fusion data. Initially, we may consider the fusion data to be commutative and half-Frobenius (refer to Remark \ref{rem:CoHalf}). 

%For ranks up to $12$, there are exactly $10628$ half-Frobenius integral fusion rings (of which $213$ are noncommutative), from $71$ types (see \S \ref{sec:Half} for details). At rank $13$, under the additional constraints, there remain only $15$ types to consider, see \S \ref{sec:R13}.

\subsection{S-matrix}
\label{sub:Smat}
Consider a commutative fusion data $(N_{i,j}^k)$ of rank $r$, eigentable $(\lambda_{i,j})$, and formal codegrees $(c_j)$ as defined in Definition \ref{def:eigentable}. The objective here is to identify all permutations $q$ of the set $\{1, \dots ,r\}$ such that:
\begin{itemize}
    \item $q(1)=1$,
    \item $d_{q(i)} = d_{i}$ for all $i$,
    \item The matrix $S = (\sqrt{c_1/c_j} \lambda_{i,q(j)})$ is symmetric (i.e. self-transpose).
\end{itemize}
\begin{remark} \label{rk:c1/cj}
The symmetric requirement implies that $$\sqrt{c_1/c_j} = \sqrt{c_1/c_j} \lambda_{1,q(j)} = \sqrt{c_1/c_1} \lambda_{j,q(1)} = d_j,$$ hence we can infer that $c_1/c_j=d_j^2$ for all $j$, as shown in \cite[Example 2.9]{OstR3}.
\end{remark}

If such a permutation $q$ exists (Remark \ref{rk:c1/cj} can serve as an effective necessary condition), the fusion data are referred to as \emph{self-transposable}. This property is exceedingly rare, rendering this step a potent sieve. Using the Verlinde formula, one can reconstruct the fusion data from $S$. It is important to note that we need only consider \emph{cyclotomic} fusion data, i.e. whose eigentable entries are all cyclotomic. 

\subsection{T-matrix}
\label{sub:Tmat}
For the remaining fusion rings $\mathcal{R}$ with fusion data $(N_{i,j}^k)$, we address the Anderson-Moore-Vafa equations:
$$\left(\sum_{p=1}^r N_{i,j}^p N_{p,k}^l \right) (t_i + t_j + t_k + t_l) = \sum_{p=1}^r \left(N_{i,j}^p N_{p,k}^l + N_{i,k}^p N_{j,p}^l + N_{j,k}^p N_{i,p}^l \right) t_p$$
within the $\mathbb{Z}$-module $\mathbb{Q}/\mathbb{Z}$. For each valid solution $t=(t_i) \in (\mathbb{Q}/\mathbb{Z})^r$, if any, the corresponding $T$-matrix is $\mathrm{diag}(e^{2 \pi \mathbf{i} t_i})$.

The solutions to the aforementioned equations are determined using the following method: Initially, the matrix reformulation is represented as $At = 0$, where $A$ is an $m \times n$ matrix over $\mathbb{Z}$ with $m=r^4$ and $n=r$. Subsequently, the Smith normal form is employed, denoted as $D = UAV$, in which $U$ and $V$ are invertible matrices over $\mathbb{Z}$ of sizes $m \times m$ and $n \times n$, respectively, and $D$ is a diagonal $m \times n$ matrix $(\alpha_i \delta_{i,j})$, where the integer $\alpha_i$ is divisible by $\alpha_{i+1}$ for all $i<r$, and $\delta_{i,j}$ is the Kronecker delta. The solutions to $Dx = 0$ are precisely represented by the vectors $(k_i/\alpha_i)$, where $0 \le k_i < \alpha_i$. Consequently, we have $U^{-1} D V^{-1} t = 0$, which simplifies to $D V^{-1} t = 0$. Therefore, the solutions can be expressed as $t = Vx$.

A complete listing of potential $T$-matrices requires considering all vectors $(k_i/\alpha_i)$, where $0 \leq k_i < \alpha_i$. As a result, there are $p=\prod_i \alpha_i$ possible combinations. This task remains manageable up to rank $11$. However, in certain case of rank above,
%(for example, one of the type $[1, 1, 1, 1, 2, 8, 18, 18, 24, 36, 36, 36]$)
the value of $p$ becomes too large.
%(e.g., $2^{23} 3^{14} \approx 4\cdot 10^{13}$). 
But a miraculous circumstance arises (referred to as the \textbf{magic criterion}): for all such cases, if one abstractly considers the $T$-matrix with variables $(k_i)$, then for every determined $S$-matrix $S$ in \S \ref{sub:Smat}, the abstract product $(ST)^3$ consistently exhibits a zero where it should not, specifically at an entry $(i,i^*)$ for some $i$. This is because $(ST)^3/p_+ = S^2 = \mathsf{d}C$, where $C$ is the duality matrix (realizing the involution $i \to i^*$, and thus $C_{i,i^*} = 1$, non-zero), as defined in Definition \ref{def:MD}. The function \texttt{MagicCriterion} (also covered by the function \texttt{STmatrix}) can verify this. 

\begin{question} \label{Q:magic}
Can the aforementioned \emph{magic criterion} be reformulated at the level of fusion data?
\end{question}

\subsection{Models for Rank 11}
\label{sub:mod11}

This subsection benefited from helpful discussions with César Galindo and Eric Rowell. Among the eight modular data described in \cite[\S\ref{subsubR11}]{ABPPsupp}, the last one corresponds exactly to the modular fusion category of finite-dimensional representations of the affine Kac-Moody algebra of type $ D_4 $ at level 2 (see \cite{kac90}), also known as $SO(8)_2$. This can be checked using \cite{GMR24}, or a recent version of SageMath \cite{sage}:  

\begin{verbatim}
sage: D42 = FusionRing("D4", 2)                           # SO(8)_2
sage: len(D42.basis())
11                                                        # Rank
sage: [(b.q_dimension()) for b in D42.basis()]
[1, 2, 2, 2, 2, 1, 2, 2, 1, 2, 1]                         # Type
sage: [x.twist()/2 for x in D42.basis()]
[0, 7/16, 3/4, 7/16, 7/16, 0, 15/16, 15/16, 0, 15/16, 0]  # Topological spins
\end{verbatim}

We need the factor \verb*|x.twist()/2| because SageMath adopts the convention $ e^{\pi \mathbf{i} t_i} $, whereas we use $ e^{2\pi \mathbf{i} t_i} $. Up to permutation and modulo 1, we recover the expected topological spins:
\[
[0, 0, 0, 0, -1/16, -1/16, -1/16, -1/4, 7/16, 7/16, 7/16].
\]

By the proof of \cite[Theorem 3.18]{GMR24}, the basic $ \FPdims $ remain invariant under zesting. Furthermore, \cite[Theorem 5.13]{GMR24} gives a zested modular fusion category whose topological spins match those in the third-to-last example of \cite[\S\ref{subsubR11}]{ABPPsupp}:
\[
[0, 0, 0, 0, -5/16, 7/16, -5/16, -1/4, 3/16, -1/16, 3/16].
\]
This result can be verified using the script \verb*|ZestD42.sage|, shared by César Galindo, located in the \verb*|Codes/SageMath| folder of the repository \cite{data}. Note that these two modular fusion categories are not Galois conjugate, since the first has four distinct topological spins, while the second has six. The remaining six MDs are obtained through Galois conjugation. To see this, simply apply the following SageMath function to the lists above.
\begin{verbatim}
def GaloisConjugates(L):
    n=lcm([denominator(i) for i in L])
    R=[]
    for m in range(n):
        if gcd([m,n])==1:
            R.append([((n*m*i)%n)/n for i in L])
    return R
\end{verbatim}
Recall that the fusion data remains invariant under Galois conjugation; see \cite{DHW13} for further details.

\section{Egyptian Fractions with Squared Denominators} \label{sec:Egy}
%A $(q,r)$-Egyptian fraction with squared denominators is defined as a sum of the form:
%$$ q = \sum_{i=1}^r \frac{1}{s_i^2},$$
%where $q, r, s_i \in \mathbb{Z}_{\ge 1}$ and the sequence satisfies $s_1 \ge s_2 \ge \cdots \ge s_r \ge 1$. Additionally, in the context of classifying potential types of Grothendieck rings for modular integral fusion categories (or more broadly, half-Frobenius integral fusion rings), we can assume that each $s_i$ is a divisor of $s_1$, for all $i$. By repeatedly subtracting $1$ from both $q$ and $r$ as necessary, we can further assume that $s_i \ge 2$, for all $i$.
%%(and so $q$ is no more assumed square-free). 
%Subsequently, we can complete the list of $(q,r)$-Egyptian fractions with squared denominators by including the $(q-k,r-k)$ ones, augmented by $k$ times the number $1$ in their sum. Using this technique, we can assume that $q \le r/4$.

A $(q,r)$-Egyptian fraction with squared denominators is a sum of the form
\[
q = \sum_{i=1}^r \frac{1}{s_i^2},
\]
where $q, r, s_i \in \mathbb{Z}_{\ge 1}$ and $s_1 \ge s_2 \ge \cdots \ge s_r \ge 1$. In the context of classifying potential Grothendieck rings of modular integral fusion categories (or more generally, half-Frobenius integral fusion rings), we can assume that each $s_i$ divides $s_1$. By removing terms equal to $1$, we can further reduce to the case $s_i \ge 2$ for all remaining terms, which immediately gives the bound $q \le r/4$. The following proposition formalizes this reduction.

\begin{proposition}\label{prop:reduction}
Consider a $(q,r)$-Egyptian fraction as defined above with $s_i \mid s_1$ for all $i$. Let $k$ be the number of indices $i$ with $s_i = 1$. Then the reduced $(q-k,r-k)$-Egyptian fraction 
\(
q - k = \sum_{i=1}^{r-k} 1/s_i^2
\) 
satisfies $s_i \ge 2$, inherits the divisibility condition, and obeys
\(
q - k \le (r - k)/4.
\) 
Conversely, any $(q,r)$-Egyptian fraction with these properties is obtained by appending $k$ unit terms $1 = 1/1^2$ to a reduced $(q-k,r-k)$-Egyptian fraction.
\end{proposition}

\begin{proof}
Since $s_1 \ge \cdots \ge s_r$, all terms equal to $1$ occur at the end. Removing these $k$ terms yields a $(q-k, r-k)$-Egyptian fraction with $s_i \ge 2$ and $s_i \mid s_1$. For these remaining terms, $1/s_i^2 \le 1/4$, so
\(
q - k = \sum_{i=1}^{r-k} 1/s_i^2 \le \sum_{i=1}^{r-k} 1/4 = (r - k)/4.
\)
The converse is immediate: appending $k$ unit terms to any reduced $(q-k,r-k)$-Egyptian fraction recovers a valid $(q,r)$-Egyptian fraction satisfying the ordering and divisibility conditions.
\end{proof}

The following steps outline our methodology for obtaining the $9025$ possible types mentioned in \S \ref{sub:types}:
\begin{itemize}
\item Employ the function \texttt{all\_rep} (available in \texttt{rep\_sq.sage} in \cite{MaxScripts}) for $1 \le r \le 13$ and $1 \le q \le r/4$.
\item Refine the classification by incorporating additional $1$s as described previously.
\item Construct all possible types using $d_i = s_1/s_i$. 
\end{itemize}

From rank 13 onward, we also used refined versions of the \texttt{all\_rep} function, as explained below, all available also in \texttt{rep\_sq.sage} in \cite{MaxScripts}. The module employs parallelization through SageMath's recursively enumerable sets (RES) functionality to enhance efficiency. The module provides the following general functions representing a given rational number as Egyptian fractions with squared denominators:
\begin{itemize}
    \item \texttt{all\_rep(q,N)}: computes all length-$N$ representations of the rational number $q$.
   % \item \texttt{count\_rep(q,N)}: counts the number of length-$N$ representations of the rational number $q$.
    \item \texttt{all\_cf\_rep(q,cf)}: computes the representations of $q$ with numerators specified by the list cf.
   % \item \texttt{count\_cf\_rep(q,cf)}: counts the representations of $q$ with numerators specified by the list $cf$.
\end{itemize}

Additionally, the module offers specialized functions for computing representations that align with the fusion ring types discussed in the paper:
\begin{itemize}
    \item \texttt{all\_rep\_MNSD(q,N)}: same as above under MNSD constraint, see \S \ref{sec:MNSD}.
    %\item \texttt{count\_rep\_MNSD(q,N)}: same as above under MNSD constraint.
    \item \texttt{all\_rep\_th87(N)}: computes the length-$N$ representations satisfying Theorems \ref{thm:RoSc} and \ref{thm:StrongPrime}.
    \item \texttt{all\_rep\_above\_r(N)}: same as above where at least one denominator has prime factor $> N$, see \S \ref{sub:ArCo}.
    \item \texttt{all\_rep\_above\_half\_r(N)}: same as above where at least one denominator has prime factor $> N/2$.
    \item \texttt{all\_rep\_primes23(N)}: same as above where denominators are restricted to have prime factors $2$ and $3$ only.
\end{itemize}

Note that functions may include additional optional parameters not listed here, which can be utilized to refine the computations further.

\section{Type Criteria}
\label{sec:crit}

In this section, we delineate criteria that were employed to exclude certain candidates from being the type of a fusion ring. A \emph{type} refers to a list denoted by $t=[[d_1,m_1],[d_2,m_2], \dots, [d_s, m_s]]$, where the conditions $1=d_1 < d_2 < \cdots < d_s$ and $m_i \ge 1$ for all indices $i$ are satisfied. Such a type is characterized as:
\begin{itemize}
\item \emph{trivial} if $t=[[1,1]]$,
\item \emph{pointed} if $t=[[1,m]]$ for some $m$,
\item \emph{perfect} if $m_1=1$,
\item \emph{integral} if each $d_i$ is an integer.
\end{itemize}
A type $t=[[d_1,m_1],[d_2,m_2], \dots, [d_s, m_s]]$ may sometimes be represented simply as $$[d_1, \dots, d_1, d_2, \dots, d_2, \dots, d_s, \dots, d_s],$$ where each $d_i$ appears $m_i$ times. Thus, we can rephrase the notation for a type of rank $r$ as $[d'_1, \dots, d'_r]$ with the condition $1 = d'_1 \le d'_2 \le \cdots \le d'_r$.
The criteria described in this section are primarily verified using modular arithmetic, with the sole exception of Proposition \ref{prop:bur}, and are presented in order of increasing computational complexity. These checks are implemented in the \verb|TypeCriteria| function within the \texttt{TypeCriteria.sage} script, as referenced in \cite{data}.  
Proposition \ref{prop:bur} is valid at the level of fusion categories but does not hold for fusion rings. Accordingly, the function \verb|TypeCriteriaRing| corresponds to \verb|TypeCriteria| with this proposition omitted. Nevertheless, the proposition is expected to remain valid for half-Frobenius fusion rings; this has been verified for all cases up to rank $13$ in \cite{BP25}.  
Applying these criteria excludes $5655$ of the $9025$ types enumerated in \S \ref{sub:types}, representing over $62\%$ of the total, and completes in approximately one minute. The breakdown of excluded types is presented in Table~\ref{tab:breakdown}.
\begin{table}[h]
\[
\begin{array}{c|ccccccccccccc}
\text{Rank} & 1 & 2 & 3 & 4 & 5 & 6 & 7 & 8 & 9 & 10 & 11 & 12 & 13 \\
\hline
\text{\# Types} & 1 & 1 & 1 & 1 & 2 & 3 & 3 & 7 & 11 & 42 & 144 & 812 & 7997 \\
\hline
\text{\# Excluded Types} & 0 & 0 & 0 & 0 & 0 & 1 & 1 & 3 & 5 & 28 & 90 & 535 & 4992
\end{array}
\]
\caption{Number of excluded types per rank}
\label{tab:breakdown}
\end{table}
For the remaining types, we will utilize the fusion ring solver, as elaborated in \S \ref{sec:FRSolver}.

\subsection{Small Perfect Type}

In this subsection, we investigate perfect fusion rings under two constraints: either the number of distinct basic $\FPdim$s is small, or the global $\FPdim$ has few prime factors.

\begin{theorem} \label{thm:small}
Any perfect integral fusion ring of type $[[d_1,m_1],[d_2,m_2], \dots, [d_s, m_s]]$, $s \le 3$, is trivial.
\end{theorem}
\begin{proof}
See the arXiv versions of this paper (up to v7) for details.
\end{proof}

Consequently, every non-trivial perfect integral fusion ring has rank at least $4$, since a fusion ring of rank at most $3$ has at most three distinct basic $\FPdim$s, and is therefore trivial by Theorem \ref{thm:small}. This theorem cannot be extended to $s=4$, as demonstrated by $\Rep(A_5)$, the representation category of the alternating group $A_5$, which has type $[[1,1],[3,2],[4,1],[5,1]]$. The proof of this theorem is deferred to the forthcoming paper \cite{BP25}, where it is more suitably placed. It effectively rules out only the following four types among the $9025$ presented in \S \ref{sec:Egy}, $$[[1, 1], [2, 2], [3, 3]], \ [[1, 1], [2, 6], [5, 3]], \ [[1, 1], [2, 2], [3, 7]], \ [[1, 1], [3, 7], [4, 5]].$$
All of these types are also excluded at the categorical level by Proposition \ref{prop:bur}, since their $\FPdim$ are of the form $p^a q^b$.

\begin{proposition} \label{prop:bur}
There exists no non-trivial perfect integral fusion category (over $\mathbb{C}$) with $\FPdim = p^a q^b$.
\end{proposition}
\begin{proof}
This follows directly from \cite[Theorem 1.6 and Proposition 4.5(iv)]{ENO11}.
\end{proof}

We note that Proposition \ref{prop:bur} does not generalize to fusion rings, as there exists a perfect fusion ring with $\FPdim = 143 = 11 \cdot 13$ and type $[1,4,4,5,6,7]$. This example is not $1$-Frobenius. Whether the proposition extends to the $1$-Frobenius case remains an open question. For reference, the extension holds in the half-Frobenius case up to rank $13$ \cite{BP25}. Theorem \ref{thm:small} and Proposition \ref{prop:bur} serve as criteria checked by the function \texttt{SmallPerfect}:

\begin{verbatim}
sage: l=[1, 2, 2, 3, 3, 3]
sage: SmallPerfect(l)
False
\end{verbatim}

\subsection{Gcd Criterion} \label{sub:gcd}

\begin{lemma}
\label{lem:gcd0}
Consider a non-pointed fusion ring of type $[d_1,d_2, \ldots, d_r]$. For all $i$ such that $d_i > 1$, let $Z_i$ be the set of indices $j \neq 1$ for which $N_{i,i^*}^j$ is nonzero, and let $g_i$ be $\gcd_{j \in Z_i}(d_j)$. Then it holds that $d_i^2 \equiv 1 \pmod{g_i}$ and $\gcd(d_i, g_i) = 1$.
\end{lemma}

\begin{proof}
First, note that $Z_i$ is non-empty, which implies that $g_i \neq 0$. According to the Frobenius-Perron theorem, the dimension equation, and the Dual axiom, we have
$$d_i^2 = d_i d_{i^*} = \sum_k d_k N_{i,i^*}^k = 1 + \sum_{j \in Z_i} d_j N_{i,i^*}^j = 1 + Kg_i,$$
where $K$ is some integer. Consequently, $d_i^2 \equiv 1 \pmod{g_i}$, and $0 \equiv 1 \pmod{\gcd(d_i, g_i)}$. The lemma follows.
\end{proof}

\begin{proposition}
\label{prop:gcd1}
Consider a non-trivial perfect fusion ring of type $[d_1,d_2, \ldots, d_r]$. Take $i>1$, let $Z'_i$ be the set of indices $j \neq 1$ for which $d_j < d_i^2$, and let $g'_i$ be $\gcd_{j \in Z'_i}(d_j)$. Then $g'_i = 1$. In particular, $\gcd(d_2, \ldots, d_r)=1$.
\end{proposition}

\begin{proof}
Note that if $N_{i,i^*}^j$ is nonzero, then $d_i^2 \ge d_j$. Hence, following the notation in Lemma \ref{lem:gcd0}, $Z_i$ is included in $Z'_i$, and as a result, $g'_i$ divides $g_i$. Due to perfectness, we have $d_i > 1$, implying $d_i^2 > d_i$ and therefore $i$ belongs to $Z'_i$. Consequently, $g'_i$ divides $d_i$. However, according to Lemma \ref{lem:gcd0}, $g'_i = 1$. For the final assertion, note that $\gcd(d_2, \ldots, d_r)$ is a divisor of $g'_2 = 1$.
\end{proof}

The SageMath code implementing the criterion from Proposition \ref{prop:gcd1} is located in the function \texttt{GcdCriterion} within the file \texttt{TypeCriteria.sage}, available at \cite{data}. This criterion eliminates over 37\% of the perfect types listed in \S \ref{sec:Egy}, such as
\begin{verbatim}
sage: l=[1, 2, 2, 2, 6, 14, 14, 21, 21, 21]
sage: GcdCriterion(l)
False
\end{verbatim}
The count per rank is presented in Table~\ref{tab:gcd}.
\begin{table}[h]
\[
\begin{array}{c|cccccc}
\text{Rank} & 8 & 9 & 10 & 11 & 12 & 13 \\
\hline
\text{\# Excluded Perfect Types} & 1 & 1 & 7 & 19 & 212 & 2474
\end{array}
\]
\caption{Number of excluded types per rank for the gcd criterion}
\label{tab:gcd}
\end{table}
\subsection{Type Test}
\label{sub:TypeTest}
Let's consider a type $t=[d_1, \dots , d_r]$ with $1=d_1 \le \cdots \le d_r$ and $d_2>1$ (signifying that it is perfect). If there is an index $i$ and $g_i>1$ such that $g_i$ divides every $d_j$ not equal to $1$ or $d_i$, and $d_i$ is coprime with $g_i$, then assume a fusion ring of this type exists with a basis $\{b_1, \dots, b_r\}$ where $d_k = \FPdim(b_k)$. 

\begin{lemma}
\label{lem:TypeTest}
For every $j$ with $d_j \neq 1$ and $d_j \neq d_i$, the following equation holds:
$$\sum_{k; \ d_k = d_i} N_{j,j^{*}}^k \equiv -1/d_i \mod g_i.$$
\end{lemma}
\begin{proof}
For each $j$ with $d_j \neq 1$ and $d_j \neq d_i$, we have:
$$ b_j b_{j^*} = b_1 + \sum_{k; \ d_k = d_i} N_{j,j^{*}}^k b_k + \sum_{k; \ d_k \neq 1, d_i} N_{j,j^{*}}^k b_k.$$
By applying $\FPdim$ and reducing modulo $g_i$, we obtain:
$$0 = 1 + xd_i \mod g_i,$$
where $d_i$ has a multiplicative inverse modulo $g_i$. Therefore, $x \equiv -1/d_i \mod g_i$.
\end{proof}

Given an integer $a_{d_i}$ such that $0 \le a_{d_i} < g_i$ and $a_{d_i} \equiv -1/d_i \mod g_i$, let $S$ be the set containing all such $d_i$. From Lemma \ref{lem:TypeTest}, for every $j \neq 1$, the inequality below must hold:
$$ d_j^2 \ge 1 + \sum_{d \in S \setminus \{d_j\}} a_d d, $$
thus if the inequality does not hold, $t$ cannot be a type of a fusion ring. Furthermore, if the set $\{k \mid d_k = d_j\}$ is a singleton, we can use a stronger inequality:
$$d_j^2 \ge 1 + b_jd_j + \sum_{d \in S \setminus \{d_j\}} a_d d, $$
with $0 \le b_j < g_j^2$ and $b_j\equiv d_j - \frac{1}{d_j} \mod g_j^2$. 

The SageMath code implementing this criterion can be found in the function \texttt{TypeTest} within the file \texttt{TypeCriteria.sage}, available at \cite{data}. This criterion eliminates over 36\% of the perfect types listed in \S \ref{sec:Egy}, such as 
\begin{verbatim}
sage: T=[[1,1],[5,1],[7,1],[35,1],[60,1],[140,2],[210,3]]
sage: TypeTest(T)
False
\end{verbatim}
The count per rank is presented in Table~\ref{tab:typetest}.
\begin{table}[h]
\[
\begin{array}{c|cccccc}
\text{Rank}  & 8 & 9 & 10 & 11 & 12 & 13   \\ \hline
\#\text{Excluded Perfect Types} & 1&1&12&37&249&2380 
\end{array}
\]
\caption{Number of excluded types per rank for the type test}
\label{tab:typetest}
\end{table}
\subsection{Local Criterion} \label{sub:LocalCrit}
Consider a type $t=[[d_1,m_1],[d_2,m_2], \dots, [d_s, m_s]]$. Assume the existence of $g, i_0>1$ such that $g$ divides each $d_i$ for all indices $i$ not in the set $\{1,i_0\}$, and $d_{i_0}$ is coprime with $g$. Let $(d,m):=(d_{i_0},m_{i_0})$. If it corresponds to a fusion ring with a basis $\{b_{1-m_1}, \dots, b_0,b_1, \dots, b_{r-1}\}$, where $b_0$ is the unit, $\FPdim(b_i) = 1$ for $i \le 0$, and $\FPdim(b_j) = d$ for $j \in \{1, \dots, m\}$, then the following lemma applies:

\begin{lemma}
\label{lem:g^2}
For each $i \in \{1, \dots, m\}$, the equation below is valid:
$$\sum_{j,k=1}^m N_{i,j}^k \equiv md - \frac{m_1}{d} \mod g^2,$$
and for all $j>m$, the integer $g$ divides $\sum_{k=1}^m N_{i,j}^k$.
\end{lemma}
\begin{proof}
For any $i \in \{1,\dots, m\}$ and $j>m$, since $\FPdim(b_i) \neq \FPdim(b_j)$, by Corollary \ref{cor:dim1} and Frobenius reciprocity, we have:
$$ b_i b_j  = \sum_{k \ge 1} N_{i,j}^k b_k = \sum_{k=1}^m N_{i,j}^k b_k + \ldots,$$
Applying $\FPdim$ and reducing modulo $g$, we conclude that:
$$ d \sum_{k=1}^m N_{i,j}^k \equiv 0 \mod g,$$
which implies that $g$ divides $\sum_{k=1}^m N_{i,j}^k$. For each $i \in \{1,\dots, m\}$, the sum over the basis elements yields:
$$b_{i^*} \sum_{k=1}^m b_k = \sum_{s \le 0} b_s + \sum_{j=1}^m \left(\sum_{k=1}^m N_{i,j}^k\right) b_j + \sum_{j>m} \left(\sum_{k=1}^m N_{i,j}^k\right) b_j.$$
After applying $\FPdim$, we obtain $md^2 = m_1 + xd + yg^2 $, hence $ x \equiv md - \frac{m_1}{d} \mod g^2$.
\end{proof}

For a type, we can analyze the partitions of $md^2 - xd - m_1$ in the form $\sum_{i \not \in \{1,i_0\}} a_i d_i$, with $x \equiv md - \frac{1}{d} \mod g^2$ and $a_i \equiv 0 \mod g$. The SageMath code performing this analysis can be found in the function \texttt{LocalCriterion} within the file specified earlier, also available at \cite{data}. This criterion, which can rule out types when no suitable partitions are found, is further detailed in the following example.

\begin{verbatim}
sage: T=[[1,1],[1295,2],[3990,1],[4218,1],[24605,1],[42180,1],[98420,2],[147630,3]]
\end{verbatim}

We can apply Lemma \ref{lem:g^2} to the triples $(d,m,g) = (1295,2,19), (3990,1,37), (4218,1,5)$. Subsequently, we obtain $md - \frac{1}{d} \equiv 126, 1135, 11 \pmod{g^2}$ for each respective triple. The application of the function \texttt{LocalCriterion} to the triple $(d,m,g) = (1295,2,19)$ enables us to eliminate the type $T$ in less than $0.1$ second.
%(written in \S \ref{subsub:sage1})

\begin{verbatim}
sage: %time LocalCriterion(T, 1295, 2, 19)
CPU times: user 82.8 ms, sys: 0 ns, total: 82.8 ms
Wall time: 82.3 ms
[]
\end{verbatim}

However, we cannot employ the triple $(d,m,g) = (3990,1,37)$, as it yields $55$ solutions.

\begin{verbatim}
sage: L = LocalCriterion(T, 3990, 1, 37); len(L)
55
\end{verbatim}

The function \texttt{LocalCriterionAll} manages all operations. It requires only two inputs: a type and a time limit.

\begin{verbatim}
sage: %time LocalCriterionAll([[1, 1], [2, 2], [3, 3], [6, 4]],1)
CPU times: user 1.44 ms, sys: 4.4 ms, total: 5.84 ms
Wall time: 9.55 ms
False
\end{verbatim}

Its application to the list in \S \ref{sec:Egy} led to the exclusion of several types per rank, as presented in Table~\ref{tab:local}.
\begin{table}[h]
\[
\begin{array}{c|cccccccc}
\text{Rank}  & 6 & 7 & 8 & 9 & 10 & 11 & 12 & 13   \\
\hline
\text{\# Excluded Types} & 1 & 1 & 3 & 5 & 21 & 63 & 344 & 2852     \\
\hline 
\text{\# Excluded Perfect Types} & 1 & 1 & 2 & 2 & 14 & 37 & 238 & 2173 \\
\end{array}
\]
\caption{Number of excluded types per rank for the local criterion}
\label{tab:local}
\end{table}
It is noteworthy that this criterion alone suffices to eliminate all perfect types up to rank $9$. Therefore, it can be stated conclusively that no non-trivial perfect integral half-Frobenius fusion rings, and thus no non-trivial perfect modular integral fusion categories, exist up to rank $9$. The use of a fusion ring solver as detailed in \S \ref{sec:FRSolver} can extend this conclusion up to rank $12$, as discussed in \S \ref{sec:Half}, and then to rank $13$ in the categorical case, see \S \ref{sec:R13}.

\section{Enhanced Fusion Ring Solver Using Normaliz}
\label{sec:FRSolver}
A fusion ring solver is a computational tool designed to receive a particular type as input and output all corresponding fusion rings of that type. Initially, \S \ref{sub:user} provides a brief introduction to the highly intuitive user interface of Normaliz \cite{Norma} from version 3.10.2, and \S \ref{sub:Norma} offers an overview of Normaliz's goals and outlines the adjustments made to support the linear and polynomial constraints specific to fusion rings. The last two subsections \S \ref{sub:full} and \S \ref{sub:DimPar} contains the approach combining SageMath and Normaliz \emph{before} version 3.10.2 (and show the system of equations explicitly). They introduce two versions of a fusion ring solver: the full version, which is discussed in \S \ref{sub:full} and addresses both dimension equations and associativity equations, and the partition (intermediate) version, which is detailed in \S \ref{sub:DimPar} and focuses on a simplified set of dimension equations through the implementation of a partition.
\subsection{Normaliz User Interface for Fusion Rings}
\label{sub:user}

Starting from version 3.10.2, Normaliz \cite{Norma} offers a streamlined user interface for computing fusion rings. This is illustrated through the input file named \verb*|bracket_4.in|, found in the \verb*|example| directory of the Normaliz distribution:

\footnotesize
\begin{verbatim}
amb_space auto
fusion_type
[1,1,2,3,3,6,6,8,8,8,12,12]
fusion_duality
[0,1,2,3,4,5,6,7,8,9,11,10]
\end{verbatim}
\normalsize

It's important to note that in the duality, indices start from 0. From this input, Normaliz generates the linear and quadratic equations defining the fusion data. The default computational goal for this input is \verb*|FusionRings|.

To run this command on Linux or MacOS, use the following command line syntax, assuming \verb*|example| is the current directory. For some progress information on the terminal, you can add the \verb*|-c| flag. On a modern laptop, the computation typically takes less than 10 seconds and requires about 2.4 GB of RAM.

\footnotesize
\begin{verbatim}
path/to/normaliz bracket_4
\end{verbatim}
\normalsize

A brief informal explanation of the algorithm used to solve the system of equations is presented in \S \ref{sub:Norma}.

The results are detailed in the \verb*|bracket_4.out| file, beginning with a preamble:

\footnotesize
\begin{verbatim}
148 fusion rings up to isomorphism
0 simple fusion rings up to isomorphism
148 nonsimple fusion rings up to isomorphism
Embedding dimension 231
dehomogenization
0 0 0 0 0 0 0 0 0 0 0 0 0 0 0 0 0 0 0 0 0 0 0 0 ... 0 0 0 0 0 1
\end{verbatim}
\normalsize

The $148$ fusion rings correspond to the orbits of the set of lattice points with respect to the symmetries of the equation system. These symmetries are observed under permutations of the type vector that adhere to the Frobenius-Perron equations and are compatible with the duality. Put simply, we have identified $148$ pairwise nonisomorphic fusion rings, classified according to their type and duality. They are automatically categorized into simple and nonsimple fusion rings. To limit the computation exclusively to simple fusion rings, one can modify the input file by including the \verb*|SimpleFusionRings| option.

The term \emph{embedding dimension} refers to the number of coordinates utilized during the computation. The rationale behind the selection of these coordinates is discussed in \S \ref{sub:full}. The final component, represented by the number $1$ in the dehomogenization process, signifies that the equations' right-hand side corresponds to the last coordinate of the solutions. This component is not included in the fusion data.

The latter part of the output file details the fusion rings represented by lattice points:

\footnotesize
\begin{verbatim}
0 simple fusion rings up to isomorphism:

148 nonsimple fusion rings up to isomorphism:
0 0 0 0 0 0 0 0 0 0 1 0 0 0 0 0 0 0 0 0 1 0 0 0 ... 1 3 3 1 1 1
...
\end{verbatim}
\normalsize

To generate fusion data from the lattice points, add the \verb*|FusionData| option in the input file. The result includes a list of fusion data for each fusion ring, presented as a series of matrices $M_i$, where $i=1,\dots,r$ and $r$ is the rank of the fusion ring ($r=12$ in our example). The matrix $M_i$ comprises the elements $N_{ij}^k$, with row index $k$ and column index $j$.

The input file for solving the dimension partition version (see \S \ref{sub:DimPar}), for instance, \verb*|bracket_3_part.in|, is as follows:

\footnotesize
\begin{verbatim}
amb_space auto
fusion_type_for_partition
[1,1,2,3,3,6,6,8,8,8,12,12]
\end{verbatim}
\normalsize

Here, the default computation goal is \verb*|SingleLatticePoint|, focusing on the solvability of the system.

For additional information and further options, refer to Appendix H of the Normaliz manual (\verb*|Normaliz.pdf|), available in the \verb*|doc| directory of the Normaliz distribution or online at \cite{NorManual}. This includes details on restricting computations to fusion rings that meet certain criteria for modular categorification.
  
\subsection{Normaliz and its Approach to Fusion Rings}
\label{sub:Norma}
Normaliz \cite{Norma} is an open source software for discrete convex geometry and its algebraic aspects. Readers are referred to Bruns and Gubeladze \cite{BrGu} for detailed terminology and a comprehensive discussion. Normaliz is designed to solve Diophantine systems of linear inequalities, equations, and congruences with integer coefficients. Additionally, it calculates enumerative information such as multiplicities (which correspond to geometric volumes) and Hilbert series. Objects in Normaliz can be defined either by generators, such as the extreme rays of cones, bases of lattices, and vertices of polytopes, or by constraints like inequalities, equations, and congruences. For systems with coefficients in real algebraic number fields, Normaliz can execute fundamental operations like convex hull computation and its dual, vertex enumeration. Moreover, it is capable of computing lattice points within (bounded) polytopes over real algebraic number fields, facilitating applications to non-integral fusion rings. In the context of fusion rings, it is crucial that lattice points within polytopes can be subjected to constraints imposed by polynomial equations and inequalities. Each release of Normaliz includes source code, comprehensive documentation, sample examples, a testing suite, and pre-compiled binaries for Linux, Mac OS, and MS Windows systems.

For lattice points in generic polytopes denoted by $P$, Normaliz employs the project-and-lift algorithm. It sequentially projects $P$ onto coordinate hyperplanes until reaching zero dimensions and then lifts the lattice points back up. If $P'$ is a projection of $P$ onto a coordinate hyperplane, then the lattice points of $P$ are projected to lattice points in $P'$, and if $x\in P'$ is a lattice point within $P'$, its preimages are the lattice points in a line segment. Polynomial constraints can be introduced as soon as the lifting process reaches the highest coordinate present in the constraint.

In its standard form, the project-and-lift method is suitable for only minor cases of fusion rings. For satisfactory performance, the algorithm has been tailored to the special linear and polynomial constraint structure specific to fusion rings. Each linear equation is inhomogeneous with nonnegative coefficients and a positive right-hand side. We can refer to the set of coordinates that appear in the equation with positive coefficients as a "patch". These patches encompass the entire set of coordinates, and thus the linear equations, when restricted to the nonnegative orthant, delineate a polytope $P$. Solutions to a linear equation, confined to its patch, are ascertained using the project-and-lift technique previously outlined, and the lattice points in $P$ are derived by combining these local solutions along matching components. In essence, we begin with the solutions of one of the equations and progressively extend them patch by patch. The sequence in which patches are integrated into the extension process is pivotal. Normaliz includes options that allow alteration of the sequence, as detailed later on.

In the partition version (\S \ref{sub:DimPar}), the input file is only made of simplified dimension (linear) equations. Particularly for this case, it is critical to recognize a secondary, implicit constraint type: congruences extracted from the linear equations by taking successive residue classes modulo their coefficients. By default, each congruence involves only the coordinates pertaining to the patch of its originating equation. Nonetheless, since congruences only involve a subset of these coordinates, they frequently pertain to other patches or combinations thereof, potentially significantly limiting their number of solutions. Our current classification up to rank 13 would not have been achievable without meticulous utilization of the congruences.

When polynomial equations of degree two or higher are in play, Normaliz endeavors to determine an optimal patch extension order that allows these equations to be applied as early as feasible. Users can influence this order by either insisting on the "linear" input order or by directing Normaliz to employ "weights" that gauge the anticipated solution count for each patch and prioritize those with lower weight. Regardless of whether polynomial equations are present, users can request an order based on the applicability of congruences. This order can also be weight-dependent.

Some computations were executed on the high-performance cluster (HPC) at Osnabrück by early splitting of partial solutions into parts, which were then processed separately. Despite the rather basic approach of using a static subdivision without intercommunication between running instances of Normaliz, the HPC proved to be advantageous.

\subsection{Full Version}
\label{sub:full}

\emph{All the processes outlined in this subsection have been fully automated in Normaliz \cite{Norma} from version 3.10.2 , as outlined in \S \ref{sub:user}. Appendix H of its manual \cite{NorManual} specifically addresses the computation of the fusion rings for a specified type.}

Consider a fusion ring with the basis $\{b_1, \dots, b_r\}$. As described in \S\ref{sub:Fu}, for all indices $i,j$:
$$ b_i b_j = \sum_k N_{i,j}^k b_k,$$
and by applying $\FPdim$, we obtain the type $[d_1, \dots, d_r]$ and the corresponding \emph{dimension equations}:
$$ d_i d_j = \sum_k N_{i,j}^k d_k.$$
The objective is to resolve these $r^2$ linear positive Diophantine equations, where $(d_i)$ are specified and $(N_{i,j}^k)$ represent $r^3$ variables, using Normaliz. Now, we can decrease the variable count to roughly $(r-1)^3/6$ by invoking the Unit axiom ($N_{1,i}^j = N_{i,1}^j = \delta_{i,j}$) from the Definition \ref{def:fu} of fusion data, as well as the Frobenius reciprocity (Proposition \ref{prop:FrobRec}).%, which can be expanded as:function $$N_{i,j}^k = N_{i^*,k}^j = N_{j,k^*}^{i^*} = N_{j^*, i^*}^{k^*} = N_{k^*, i}^{j^*} = N_{k, j^*}^{i}.$$

A critical factor in accelerating computation is the strategic use of associativity equations (non-linear)
$$\sum_s N_{i,j}^s N_{s,k}^t = \sum_s N_{j,k}^s N_{i,s}^t,$$
in the most effective manner possible during the solving process of the aforementioned linear Diophantine equations. While the optimal approach is not confirmed, the \emph{patching} method we employ is highly efficient (refer to \S\ref{sub:Norma} for further details).

In practice, for a given type $L=[d_1, d_2, \dots, d_r]$, utilize the function \texttt{TypeToNormaliz}, the SageMath code for which can be found at \cite{data}. This function generates input files (.in), one for each potential duality map $i \to i^*$. Place these files in a directory alongside the normaliz.exe and run\_normaliz.bat files available at \cite{data}, and execute run\_normaliz (note the existence of a more recent and faster Linux version used for our latest computations). This process yields output files (.out) containing all potential solutions (if any exist). The remaining task is to convert these solutions into fusion data, considering isomorphism. We demonstrate how this can be done with the following example. Take the type $L=[1,1,2]$ of the character ring of $S_3$. When \texttt{TypeToNormaliz} is applied, it generates the file [1,1,2][0,1,2].in with the content as follows:

\footnotesize
\begin{verbatim}
amb_space 4
inhom_equations 4
1 2 0 0 0 
0 1 2 0 -2 
0 1 2 0 -2 
0 0 1 2 -3 
LatticePoints
convert_equations
nonnegative
polynomial_equations 2
x[2]^2 - x[1]*x[3] + x[3]^2 - x[2]*x[4] - 1;
-x[2]^2 + x[1]*x[3] - x[3]^2 + x[2]*x[4] + 1;
\end{verbatim}
\normalsize

The upper part encodes the linear Diophantine equations, and the lower part lists the associativity equations. Following the execution of run\_normaliz, the file [1,1,2][0,1,2].out is produced, containing:

\footnotesize
\begin{verbatim}
1 lattice points in polytope (module generators) satisfying polynomial constraints:
 0 0 1 1 1
\end{verbatim}
\normalsize

Here, we encounter a single solution, but there could be multiple in general (as seen in the subsequent example). Next, remove the final '1' from each line of the solution and convert it into a list of lists:

\begin{verbatim}
sage: LL=[[0,0,1,1]]
\end{verbatim}

Collect the lists for the type and the duality map:

\begin{verbatim}
sage: L=[1,1,2]
sage: d=[0,1,2]
\end{verbatim}

Finally, to obtain all the fusion data up to isomorphism, apply the function \texttt{ListToFusion}:

\begin{verbatim}
sage: ListToFusion(LL,L,d)
[[[[1, 0, 0], [0, 1, 0], [0, 0, 1]],
  [[0, 1, 0], [1, 0, 0], [0, 0, 1]],
  [[0, 0, 1], [0, 0, 1], [1, 1, 1]]]]
\end{verbatim}

The result is the fusion data of $\ch(S_3)$, which can ultimately be formatted in TeX as follows:

 $$\begin{matrix}1&0&0 \\ 0&1&0 \\ 0&0&1\end{matrix} \ , \  \ \ \ \begin{matrix}0&1&0 \\ 1&0&0 \\ 0&0&1\end{matrix} \ , \ \ \ \ \begin{matrix}0&0&1 \\ 0&0&1 \\ 1&1&1 \end{matrix}$$

Now, applying the same procedure with the type $L=[1,5,5,5,6,7,7]$, we obtain four input files. Only the file corresponding to the trivial duality map yields solutions, with its output file containing:

\footnotesize
\begin{verbatim}
6 lattice points in polytope (module generators) satisfying polynomial constraints:
 1 0 1 0 1 1 1 0 1 1 1 0 1 1 1 1 1 1 1 1 1 1 0 0 1 1 1 1 1 1 1 1 1 1 1 1 1 0 1 1 1 1 1 1 1 1 1 1 1 2 1 2 0 3 1 2 1
 1 0 1 0 1 1 1 0 1 1 1 0 1 1 1 1 1 1 1 1 1 1 0 0 1 1 1 1 1 1 1 1 1 1 1 1 1 0 1 1 1 1 1 1 1 1 1 1 1 2 1 2 1 2 2 1 1
 1 0 1 0 1 1 1 0 1 1 1 0 1 1 1 1 1 1 1 1 1 1 0 0 1 1 1 1 1 1 1 1 1 1 1 1 1 0 1 1 1 1 1 1 1 1 1 1 1 2 1 2 2 1 3 0 1
 1 1 0 0 1 1 0 0 1 1 1 1 1 1 1 1 1 1 1 1 1 1 1 0 1 1 0 1 1 1 1 1 1 1 1 1 1 0 1 1 1 1 1 1 1 1 1 1 1 2 1 2 0 3 1 2 1
 1 1 0 0 1 1 0 0 1 1 1 1 1 1 1 1 1 1 1 1 1 1 1 0 1 1 0 1 1 1 1 1 1 1 1 1 1 0 1 1 1 1 1 1 1 1 1 1 1 2 1 2 1 2 2 1 1
 1 1 0 0 1 1 0 0 1 1 1 1 1 1 1 1 1 1 1 1 1 1 1 0 1 1 0 1 1 1 1 1 1 1 1 1 1 0 1 1 1 1 1 1 1 1 1 1 1 2 1 2 2 1 3 0 1
\end{verbatim}
\normalsize

All files can be accessed at \cite{data}. Ultimately, we acquire the following two sets of fusion data, up to isomorphism:

$$ \begin{smallmatrix}1&0&0&0&0&0&0 \\ 0&1&0&0&0&0&0 \\ 0&0&1&0&0&0&0 \\ 0&0&0&1&0&0&0 \\ 0&0&0&0&1&0&0 \\ 0&0&0&0&0&1&0 \\ 0&0&0&0&0&0&1\end{smallmatrix} ,   \ \begin{smallmatrix}0&1&0&0&0&0&0 \\ 1&1&0&1&0&1&1 \\ 0&0&1&0&1&1&1 \\ 0&1&0&0&1&1&1 \\ 0&0&1&1&1&1&1 \\ 0&1&1&1&1&1&1 \\ 0&1&1&1&1&1&1\end{smallmatrix} ,   \ \begin{smallmatrix}0&0&1&0&0&0&0 \\ 0&0&1&0&1&1&1 \\ 1&1&1&0&0&1&1 \\ 0&0&0&1&1&1&1 \\ 0&1&0&1&1&1&1 \\ 0&1&1&1&1&1&1 \\ 0&1&1&1&1&1&1\end{smallmatrix} ,   \ \begin{smallmatrix}0&0&0&1&0&0&0 \\ 0&1&0&0&1&1&1 \\ 0&0&0&1&1&1&1 \\ 1&0&1&1&0&1&1 \\ 0&1&1&0&1&1&1 \\ 0&1&1&1&1&1&1 \\ 0&1&1&1&1&1&1\end{smallmatrix} ,   \ \begin{smallmatrix}0&0&0&0&1&0&0 \\ 0&0&1&1&1&1&1 \\ 0&1&0&1&1&1&1 \\ 0&1&1&0&1&1&1 \\ 1&1&1&1&1&1&1 \\ 0&1&1&1&1&2&1 \\ 0&1&1&1&1&1&2\end{smallmatrix} ,   \ \begin{smallmatrix}0&0&0&0&0&1&0 \\ 0&1&1&1&1&1&1 \\ 0&1&1&1&1&1&1 \\ 0&1&1&1&1&1&1 \\ 0&1&1&1&1&2&1 \\ 1&1&1&1&2&0&3 \\ 0&1&1&1&1&3&1\end{smallmatrix} ,   \ \begin{smallmatrix}0&0&0&0&0&0&1 \\ 0&1&1&1&1&1&1 \\ 0&1&1&1&1&1&1 \\ 0&1&1&1&1&1&1 \\ 0&1&1&1&1&1&2 \\ 0&1&1&1&1&3&1 \\ 1&1&1&1&2&1&2\end{smallmatrix} $$

$$ \begin{smallmatrix}1&0&0&0&0&0&0 \\ 0&1&0&0&0&0&0 \\ 0&0&1&0&0&0&0 \\ 0&0&0&1&0&0&0 \\ 0&0&0&0&1&0&0 \\ 0&0&0&0&0&1&0 \\ 0&0&0&0&0&0&1\end{smallmatrix} ,   \ \begin{smallmatrix}0&1&0&0&0&0&0 \\ 1&1&0&1&0&1&1 \\ 0&0&1&0&1&1&1 \\ 0&1&0&0&1&1&1 \\ 0&0&1&1&1&1&1 \\ 0&1&1&1&1&1&1 \\ 0&1&1&1&1&1&1\end{smallmatrix} ,   \ \begin{smallmatrix}0&0&1&0&0&0&0 \\ 0&0&1&0&1&1&1 \\ 1&1&1&0&0&1&1 \\ 0&0&0&1&1&1&1 \\ 0&1&0&1&1&1&1 \\ 0&1&1&1&1&1&1 \\ 0&1&1&1&1&1&1\end{smallmatrix} ,   \ \begin{smallmatrix}0&0&0&1&0&0&0 \\ 0&1&0&0&1&1&1 \\ 0&0&0&1&1&1&1 \\ 1&0&1&1&0&1&1 \\ 0&1&1&0&1&1&1 \\ 0&1&1&1&1&1&1 \\ 0&1&1&1&1&1&1\end{smallmatrix} ,   \ \begin{smallmatrix}0&0&0&0&1&0&0 \\ 0&0&1&1&1&1&1 \\ 0&1&0&1&1&1&1 \\ 0&1&1&0&1&1&1 \\ 1&1&1&1&1&1&1 \\ 0&1&1&1&1&2&1 \\ 0&1&1&1&1&1&2\end{smallmatrix} ,   \ \begin{smallmatrix}0&0&0&0&0&1&0 \\ 0&1&1&1&1&1&1 \\ 0&1&1&1&1&1&1 \\ 0&1&1&1&1&1&1 \\ 0&1&1&1&1&2&1 \\ 1&1&1&1&2&1&2 \\ 0&1&1&1&1&2&2\end{smallmatrix} ,   \ \begin{smallmatrix}0&0&0&0&0&0&1 \\ 0&1&1&1&1&1&1 \\ 0&1&1&1&1&1&1 \\ 0&1&1&1&1&1&1 \\ 0&1&1&1&1&1&2 \\ 0&1&1&1&1&2&2 \\ 1&1&1&1&2&2&1\end{smallmatrix} $$

\subsection{Dimension Partition Version}
\label{sub:DimPar}

This method is applicable primarily for types denoted by $$T=[[1,m_1],[d_2,m_2], \dots, [d_s, m_s]],$$ where $s$ is not exceedingly large. This is because we can streamline the dimension equations by grouping elements that share the same dimension (i.e. dimension partition). However, the conversion of the associativity equations remains an open challenge. This version is intended to serve as an intermediary step to the full version for suitable types. Its utility lies in its ability to circumvent certain computational complexities by breaking symmetries. For the time being, it functions as a criterion; that is, if this version fails to yield a solution, the full version will similarly lack a solution.

We can reframe the type as $[1, d_{1,1}, \dots, d_{1,n_1}, d_{2,1}, \dots, d_{2,n_2}, \dots, d_{s,1}, d_{s,n_s}]$, where $d_{i,a}= d_i$, $d_1 = 1 = d_{0,1}$, and $n_i = m_i - \delta_{1,i}$. The dimension equations are then expressed as follows:
$$ d_{i,a} d_{j,b} = \sum_{k,c} N_{i,a,j,b}^{k,c} d_{k,c}.$$
Let us define $D_i:=\sum_{a=1}^{n_i} d_{i,a} = n_id_i$ and $M_{i,j}^k := \sum_{a,b,c} N_{i,a,j,b}^{k,c}$, which simplifies the equations to:
$$D_i D_j = \sum_{a,b}\sum_{k,c} N_{i,a,j,b}^{k,c} d_{k,c} = \sum_k (\sum_{a,b,c} N_{i,a,j,b}^{k,c}) d_{k}= \sum_k M_{i,j}^k d_{k}.$$
Consequently, we are tasked with solving the linear positive Diophantine equations:
$$n_i d_i n_j d_j = \sum_k M_{i,j}^k d_{k},$$
where $(d_i,n_i)$ are predetermined, and the variables $(M_{i,j}^k)$ are reduced to roughly $s^3/6$ by employing the dimension partition variant of the Unit axiom and Frobenius reciprocity. After grouping by dimension, the duality map becomes straightforward (that is, $i^* = i$). Note that we have not yet derived a satisfactory dimension partition version of the associativity axiom, but about the other ones:

\begin{lemma}
\label{lem:FuPart}
The following equalities hold:%The dimension partition version of:
\begin{itemize}
\item (Unit) $M_{i,0}^j = M_{0,i}^j =  \delta_{i,j}m_i$
\item (Dual) $M_{i,j}^0 = M_{j,i}^0 =  \delta_{i,j}m_i$
\item (Frobenius reciprocity) $M_{i,j}^k = M_{i,k}^j = M_{j,k}^{i} = M_{j, i}^{k} = M_{k, i}^{j} = M_{k, j}^{i}.$
\end{itemize}
\end{lemma}

\begin{proof}
The proof is straightforward.
\end{proof}

In practice, one should follow the procedure outlined in \S \ref{sub:full} up to the generation of output files but replace the function \texttt{TypeToNormaliz} with \texttt{TypeToPreNormaliz}. For instance, consider the type $L=[1,6,12,12,15,15,15,20,20,30,30,60]$. 

%The corresponding input and output files can be found in the reference \cite{data}. This dimension partition version is sufficiently robust to demonstrate Theorem \ref{thm:perfect12} in instances without prime-power basic $\FPdim$ (refer to \S \ref{sec:Half} immediately following Theorem \ref{thm:PrimePower}), with $L$ being the first of $24$ types to be excluded at rank $12$.

\begin{remark}
While this version utilizes the dimension partition of the type, alternative versions could explore other pertinent partitions.
\end{remark}

\section{Half-Frobenius Integral Fusion Rings up to Rank 12}
\label{sec:Half}

This section focuses on classifying all half-Frobenius integral fusion rings up to rank $12$. Initially, we considered $1028$ types derived from Egyptian fractions with squared denominators, as discussed in \S \ref{sub:types}. From these, we identified $71$ types of fusion rings, but none were perfect, thus proving Theorem \ref{thm:perfect12}. Subsequently, we classified $10628$ fusion rings originating from these $71$ types, detailed in \S \ref{sub:ClassRings}. Among them, we identified $213$ noncommutative fusion rings. Ultimately, only $69$ fusion rings, from $27$ types, are commutative, cyclotomic, and self-transposable, as outlined in \S \ref{sub:CoCyS}.

\subsection{List of Possible Types} \label{sub:types}
Based solely on Egyptian fractions with squared denominators, we found $1028$ possible types up to rank $12$ (in fact, $9025$ ones up to rank $13$). The counting per rank is presented in Table~\ref{tab:types}.
\begin{table}[h]
\[
\begin{array}{c|cccccccccccc|c}
\text{Rank} & 1 & 2 & 3 & 4 & 5 & 6 & 7 & 8 & 9 & 10 & 11 & 12 & 13  \\ \hline
\# \text{Types} & 1 & 1 & 1 & 1 & 2 & 3 & 3 & 7 & 11 & 42 & 144 & 812 & 7997 \\ \hline
\# \text{Perfect Types} & 1 & 0 & 0 & 0 & 0 & 1 & 1 & 2 & 2 & 24 & 88 & 591 & 6517
\end{array}
\]
\caption{Number of possible types per rank}
\label{tab:types}
\end{table}

The comprehensive list of possible types up to rank $13$, available in \texttt{ListPossibleTypes.tar.xz} in \cite{data}, was computed using \texttt{all\_rep} from \S \ref{sec:Egy}. The ratio of perfect types (refer to \S \ref{sec:crit}) exhibits an increasing trend, e.g. $18\%$ for rank $9$, but $81\%$ for rank $13$. This leads us to question whether this ratio tends to $1$ as the rank goes to infinity.

\subsection{List of Fusion Rings} \label{sub:ClassRings}
All computational steps are documented in \texttt{InvestHFuptoR12.txt}, as referenced in \cite{data}. The initial step involved processing the $1016$ non-pointed types up to rank $12$ mentioned in \textup{\S}\,\ref{sub:types}. From these, merely $352$ types met the \texttt{TypeCriteria} outlined in \textup{\S}\,\ref{sec:crit}. Utilizing the partition version of our fusion ring solver, as discussed in \textup{\S}\,\ref{sub:DimPar} and limiting the processing time to ten seconds per type, further reduced the count to just $77$ types. Subsequently, excluding the type $[1,1,63,135,140,252,540,1260,1260,1890,1890,1890]$ with HPC brought the total down to $76$ types. The distribution of these types across ranks $5$ to $12$ is $1, 1, 1, 3, 5, 9, 17, 39$, respectively. Out of these, only $[1,2,2,3,4,4,15,15,20,30,30,30]$ is perfect. Following this, we employed the full version of our fusion ring solver, as outlined in \textup{\S}\,\ref{sub:full}, which necessitated specifying duality. This adjustment resulted in the expansion of our $76$ types into $801$ cases. We streamlined these cases using the \texttt{SingleLatticePoint} option, as elucidated in \textup{\S}\,\ref{sub:user}. After an initial round of processing, limited to $10$ seconds per case, $50$ cases remained unresolved, plus $312$ cases with a solution. The only perfect type is already excluded at this step, which proves Theorem~\ref{thm:perfect12}. Further rounds of processing, with limits set to $100$ seconds and then $1000$ seconds, resulted in $10$ then $3$ unresolved cases, respectively. Ultimately, we solved the last cases with HPC, concluding with solutions for $340$ cases across $59$ distinct non-pointed types, as listed below for each rank, arranged in lexicographic order:

\footnotesize
\begin{itemize}
\item Rank 5: [[1, 1, 1, 1, 2]],
\item Rank 6: [[1, 1, 1, 1, 2, 2]],
\item Rank 7: [[1, 1, 1, 1, 2, 2, 2]],
\item Rank 8: [[1, 1, 1, 1, 2, 2, 2, 2], [1, 1, 1, 1, 2, 2, 2, 4], [1, 1, 2, 2, 2, 2, 3, 3]],
\item Rank 9:
 [[1, 1, 1, 1, 2, 2, 2, 2, 2],
  [1, 1, 1, 1, 2, 2, 2, 4, 4],
  [1, 1, 1, 1, 4, 4, 6, 6, 6],
  [1, 1, 2, 2, 2, 2, 3, 3, 6]],
\item Rank 10: [[1, 1, 1, 1, 1, 1, 1, 1, 1, 3],
  [1, 1, 1, 1, 1, 1, 1, 1, 2, 2],
  [1, 1, 1, 1, 2, 2, 2, 2, 2, 2],
  [1, 1, 1, 1, 2, 2, 2, 4, 4, 4],
  [1, 1, 1, 1, 4, 4, 6, 6, 6, 12],
  [1, 1, 1, 2, 2, 2, 2, 2, 2, 3],
  [1, 1, 2, 2, 2, 2, 3, 3, 6, 6],
  [1, 1, 2, 3, 3, 4, 4, 4, 6, 6]],
\item Rank 11: [[1, 1, 1, 1, 1, 1, 1, 1, 1, 3, 3],
  [1, 1, 1, 1, 2, 2, 2, 2, 2, 2, 2],
  [1, 1, 1, 1, 1, 1, 1, 1, 2, 2, 4],
  [1, 1, 1, 1, 1, 1, 2, 2, 2, 3, 3],
  [1, 1, 1, 2, 2, 2, 2, 2, 2, 3, 6],
  [1, 1, 1, 1, 2, 2, 2, 4, 4, 4, 4],
  [1, 1, 1, 1, 2, 2, 2, 4, 4, 4, 8],
  [1, 1, 1, 1, 2, 4, 4, 4, 4, 6, 6],
  [1, 1, 2, 2, 2, 2, 3, 3, 6, 6, 6],
  [1, 1, 2, 3, 3, 4, 4, 4, 6, 6, 12],
  [1, 1, 1, 1, 2, 6, 6, 8, 12, 12, 12],
  [1, 1, 1, 1, 4, 4, 6, 6, 6, 12, 12],
  [1, 1, 1, 3, 4, 4, 4, 4, 4, 4, 6],
  [1, 1, 1, 1, 4, 4, 12, 12, 18, 18, 18]],
\item Rank 12: [[1, 1, 1, 1, 1, 1, 1, 1, 1, 3, 3, 3], [1, 1, 1, 1, 1, 1, 1, 1, 2, 2, 2, 2], [1, 1, 1, 1, 1, 1, 1, 1, 2, 2, 4, 4], [1, 1, 1, 1, 1, 1, 2, 2, 2, 3, 3, 6], [1, 1, 1, 1, 2, 2, 2, 2, 2, 2, 2, 2], [1, 1, 1, 1, 2, 2, 2, 2, 2, 2, 2, 4], [1, 1, 1, 1, 2, 2, 2, 2, 4, 6, 6, 6], [1, 1, 1, 1, 2, 2, 2, 4, 4, 4, 4, 4], [1, 1, 1, 1, 2, 2, 2, 4, 4, 4, 8, 8], [1, 1, 1, 1, 2, 2, 2, 8, 8, 12, 12, 12], [1, 1, 1, 1, 2, 4, 4, 4, 4, 6, 6, 12], [1, 1, 1, 1, 2, 6, 6, 8, 12, 12, 12, 24], [1, 1, 1, 1, 2, 8, 18, 18, 24, 36, 36, 36], [1, 1, 1, 1, 3, 3, 3, 3, 4, 4, 6, 6], [1, 1, 1, 1, 4, 4, 6, 6, 6, 12, 12, 12], [1, 1, 1, 1, 4, 4, 12, 12, 18, 18, 18, 36], [1, 1, 1, 2, 2, 2, 2, 2, 2, 3, 6, 6], [1, 1, 1, 2, 2, 2, 3, 4, 4, 4, 6, 6], [1, 1, 1, 3, 4, 4, 4, 4, 4, 4, 6, 12], [1, 1, 1, 3, 6, 8, 8, 8, 8, 8, 8, 12], [1, 1, 2, 2, 2, 2, 3, 3, 3, 3, 3, 3], [1, 1, 2, 2, 2, 2, 3, 3, 6, 6, 6, 6], [1, 1, 2, 2, 2, 2, 3, 3, 6, 6, 6, 12], [1, 1, 2, 2, 2, 2, 6, 6, 6, 6, 9, 9], [1, 1, 2, 3, 3, 4, 4, 4, 6, 6, 12, 12], [1, 1, 2, 3, 3, 6, 6, 8, 8, 8, 12, 12], [1, 1, 2, 6, 6, 6, 6, 10, 10, 10, 15, 15]].
\end{itemize}
\normalsize
As previously mentioned, none are perfect. It is noteworthy that the perfect integral modular fusion category, and therefore half-Frobenius, $\mathcal{Z}(\Rep(A_5))$ has an $\FPdim$ of $60^2=3600$, a rank of $22$, and a type of $[[1, 1], [3, 2], [4, 1], [5, 1], [12, 10], [15, 4], [20, 3]]$, calculated using \cite[Section 3]{NN} and GAP.

\begin{question} \label{Q:RankLess22}
Is there a perfect integral half-Frobenius fusion ring/category with a rank less than $22$?
\end{question}

Note that the perfect integral modular fusion category $\mathcal{Z}(\Rep(A_7))$, with $\FPdim$ $(7!/2)^2$, rank $74$, and type:
$$
\left[ [ 1, 1 ], [ 6, 1 ], [ 10, 2 ], [ 14, 2 ], [ 15, 1 ], [ 21, 1 ], [ 35, 1 ], [ 70, 9 ], [ 105, 4 ], [ 210, 20 ], [ 280, 9 ], [ 360, 14 ], [ 504, 5 ], [ 630, 4 ] \right],
$$
notably lacks any prime-power basic $\FPdim$.

\begin{question} \label{Q:less74}
Is there a perfect integral half-Frobenius fusion ring/category, without any prime-power basic $\FPdim$, that has a rank lower than $74$?
\end{question}

We ended with $71$ types (pointed included), and the computation with the \verb*|FusionData| option (retricted to the $340$ cases mentioned above, in the non-pointed case) provides $10628$ fusion rings, $213$ among them being noncommutative. The counting per rank is presented in Table~\ref{tab:halfFR}.
\begin{table}[h]
\[
\begin{array}{c|cccccccccccc}
\text{Rank} & 1 & 2 & 3 & 4 & 5 & 6 & 7 & 8 & 9 & 10 & 11 & 12  \\ \hline
\# \text{Types} & 1 & 1 & 1 & 1 & 2 & 2 & 2 & 4 & 5 & 9 & 15 & 28 \\ \hline
\# \text{Fusion Rings} & 1 & 1 & 1 & 2 & 3 & 6 & 9 & 23 & 105 & 158 & 1218  & 9101  \\ \hline
\# \text{Noncommutative} & 0 & 0 & 0 & 0 & 0 & 1 & 0 & 4 & 5 & 7 & 38 & 158
\end{array}
\]
\caption{Number of half-Frobenius fusion rings per rank}
\label{tab:halfFR}
\end{table}
The list of half-Frobenius integral fusion rings up to rank $12$ is available in \texttt{HalfFrobIntUpToRank12.tar.xz} in \cite{data}.

\subsection{Commutative, Cyclotomic and Self-Transposable} \label{sub:CoCyS}

Among the $10628$ half-Frobenius integral fusion rings up to rank $12$ discovered above, only $99.3\%$ are commutative cyclotomic and self-transposable (see \S \ref{sub:Smat}), i.e., $69$ fusion rings from $27$ types. The counting per rank is presented in Table~\ref{tab:ccsthalfFR}.
\begin{table}[h]
\[
\begin{array}{c|cccccccccccc}
\text{Rank} & 1 & 2 & 3 & 4 & 5 & 6 & 7 & 8 & 9 & 10 & 11 & 12 \\ \hline
 \#  \text{Types} & 1  &  1 &  1 &  1 & 1  &  1 &  2 &  2 & 2 & 4 &  5 & 6 \\ \hline
 \#  \text{Fusion Rings} & 1  &  1 &  1 &  2 &  1 & 1  &  3 & 7  & 4 & 11 &  13 & 24
\end{array}
\]
\caption{Number of commutative, cyclotomic and self-transposable half-Frobenius fusion rings per rank}
\label{tab:ccsthalfFR}
\end{table}
The types mentioned above, restricted to the non-pointed ones, are listed below:

\begin{itemize} {\small
\item Rank 7: [1, 1, 1, 1, 2, 2, 2], 
\item Rank 8: [1, 1, 2, 2, 2, 2, 3, 3], 
\item Rank 9: [1, 1, 1, 1, 4, 4, 6, 6, 6], 
\item Rank 10:[1, 1, 1, 1, 2, 2, 2, 4, 4, 4], [1, 1, 1, 2, 2, 2, 2, 2, 2, 3], [1, 1, 2, 3, 3, 4, 4, 4, 6, 6], 
\item Rank 11:[1, 1, 1, 1, 2, 2, 2, 2, 2, 2, 2], [1, 1, 1, 1, 2, 6, 6, 8, 12, 12, 12], [1, 1, 1, 3, 4, 4, 4, 4, 4, 4, 6], [1, 1, 1, 1, 4, 4, 12, 12, 18, 18, 18], 
\item Rank 12:[1, 1, 1, 1, 2, 8, 18, 18, 24, 36, 36, 36], [1, 1, 1, 3, 6, 8, 8, 8, 8, 8, 8, 12], [1, 1, 2, 2, 2, 2, 6, 6, 6, 6, 9, 9], [1, 1, 2, 3, 3, 6, 6, 8, 8, 8, 12, 12], [1, 1, 2, 6, 6, 6, 6, 10, 10, 10, 15, 15].}
\end{itemize}
The full list of non-pointed commutative, cyclotomic, self-transposable half-Frobenius integral fusion rings up to rank $12$ is available in \texttt{NonPtCycloHalfFrobSelfTransUptoRk12.tar.xz} in \cite{data}. They were classified utilizing the list from \S \ref{sub:ClassRings} in conjunction with the functions \texttt{NonCo}, \texttt{ExtendedCyclo}, \texttt{ModularCriterion} and \texttt{SelfTransposable} contained within the file \texttt{ModularData.sage}, also available at \cite{data}. 
%At rank $13$, there remain only $6$ types with such fusion rings to consider (see the list in \S \ref{sub:NP13}).

\section{Advanced Results on Modular Fusion Categories} \label{sec:AdvMD}
From rank $13$ onwards, it becomes impractical to classify all half-Frobenius integral fusion rings using our current technology. Thus, the types were further restricted by additional properties coming from more advanced results on modular fusion categories, which involved the universal grading \S \ref{sub:univ}, congruence representations of the modular group \S \ref{sub:cong} and Galois action \S \ref{sub:gal}.

\subsection{Universal Grading} \label{sub:univ}
Let $G$ be a finite group. A \emph{$G$-grading} of a fusion ring $R$ is given by a partiton of its basis $B = \sqcup_{g \in G} B_g$ such that:
\begin{itemize}
\item For any $x \in B_g$ and any $y \in B_{g'}$, the basic components of $xy$ belong to $B_{gg'}$.
\item For any $x \in B_g$, $x^*$ is in $B_{g^{-1}}$.
\end{itemize}
A $G$-grading is called \emph{faithful} if $B_g$ is non-empty for all $g \in G$. Consequently, by \cite[Theorem 3.5.2]{EGNO}, $\FPdim(B_g):=\sum_{x \in B_g} \FPdim(x)^2$ is constant in $g$. The faithful grading with the largest group is called the \emph{universal grading}.
%\end{definition}
By \cite[Lemma 8.22.9]{EGNO}, the universal grading group of the Grothendieck ring of modular fusion category is $G=B_{pt}$, the group of basic element with $\FPdim = 1$ (see Corollary \ref{cor:pointed}). All these modular constraints lead to the following definition:

\begin{definition} \label{def:ModPart}
Let $t$ be a \emph{type}, i.e., a sorted list of integers starting with $1$. Let $r$ be the length of $t$. Let $p$ be the number of entries equal to $1$ in $t$. Let $D:=\sum_{d \in t} d^2$. A \emph{modular partition} of $t$ is a list $L$ of lists such that: 
\begin{itemize}
\item $t$ is the sorted concatenation of the lists in $L$,
\item $L$ is lexicographically sorted,
\item the lists in $L$ are sorted,
\item $L$ has $p$ elements,
\item $p$ divides $D$,
\item for all $l$ in $L$, then $\sum_{d \in l} d^2 = D/p$,
\end{itemize}
A solution for $t$ may be called a \emph{partitioned type}.
\end{definition}

The function classifying all the modular partitions of a given type is named \texttt{ModularPartitions} in \texttt{TypeCriteria.sage} in \cite{data}. Here are a few examples with $0$, $1$, $2$, and $3$ solutions:

\begin{verbatim}
sage: %attach TypeCriteria.sage
sage: L0=[1,1,1,1,2]
sage: ModularPartitions(L0)
[]
sage: L1=[1, 1, 1, 3, 12, 12, 30, 40, 40, 40, 40, 40, 40, 60]
sage: ModularPartitions(L1)
[[[1, 1, 1, 3, 12, 12, 30, 60], [40, 40, 40], [40, 40, 40]]]
sage: L2=[1, 1, 1, 1, 2, 2, 2, 2, 2, 2, 2, 4, 4, 4]
sage: ModularPartitions(L2)
[[[1, 1, 1, 1, 4], [2, 2, 2, 2, 2], [2, 4], [2, 4]],
 [[1, 1, 1, 1, 2, 2, 2, 2], [2, 4], [2, 4], [2, 4]]]
sage: L3=[1, 1, 2, 2, 2, 2, 3, 3, 3, 3, 3, 3, 6, 6]
sage: ModularPartitions(L3)
[[[1, 1, 2, 2, 2, 2, 3, 3, 6], [3, 3, 3, 3, 6]],
 [[1, 1, 2, 2, 2, 2, 3, 3, 3, 3, 3, 3], [6, 6]],
 [[1, 2, 2, 3, 3, 3, 6], [1, 2, 2, 3, 3, 3, 6]]]
\end{verbatim}

The classification of fusion rings constrained to such a grading, with grading group of order less than $5$, is a new Normaliz feature, see \cite[\S H.5.2]{NorManual}. 
%, see \S \ref{sec:FRSolver}.
The following theorem is essentially a reformulation of \cite[Proposition VI.2]{NRW23}.
\begin{theorem} \label{thm:ModCrit} Let $\mathcal{C}$ be an integral modular fusion category. Let $R$ be its Grothendieck ring with basis $B$. Let $G=B_{pt}$ be the universal grading group. Let $t_g := (\FPdim(x))_{x \in B_g}$. Let $\mathcal{C}_e$ be the fusion subcategory corresponding to $B_e$. 
Then:
\begin{itemize}
\item[(0)] If $\mathcal{C}_e$ is perfect then it is modular,
\item[(1)] If $B_{pt} \subset B_e$ and $t_e$ has an entry with odd multiplicity, then $\forall g \neq e$, every entry of $t_g$ has multiplicity $\ge 2$,
\item[(2)] If condition (1) is satisfied, $p:=|B_{pt}|$ is prime, and an entry $d$ in $t_e$ appears with multiplicity $m$, then $p$ divides $d$ or $m$,
\item[(3)] If condition (1) is satisfied, $p:=|B_{pt}|$ is prime, then there exists a modular fusion category $\mathcal{D}$ with $\FPdim = \FPdim(\mathcal{C})/p^2$ and type $t_e'$, where $t_e'$ is a reduction of $t_e$ that involves mapping $p$ identical entries $x, \dots, x$ to a single entry $x$, or alternatively, mapping one entry $x$ which is a multiple of $p$ to $p$ entries $x/p, \dots, x/p$.
\end{itemize}
\end{theorem}
%\begin{proof}
%The sentence (0) is contained in \cite[Proposition VI.2]{NRW23}. By \cite[Proposition VI.2 (b)]{NRW23}, and following the notation of this reference, condition (1) forces the invertible objects to be bosons, so $\mathcal{C}_{pt}$ is Tannakian, thus (1) follows from \cite[Proposition VI.2 (d)]{NRW23}. Finally, (3) follows from \cite[Proposition VI.2 (e)]{NRW23}, citing the \emph{modularization} of \cite{Bru00}, as a group of prime order must be cyclic; and (3) implies (2) trivially.
%\end{proof}

\begin{remark} \label{rk:modtype}
Since the Grothendieck ring of a modular fusion category $\mathcal{C}$ is half-Frobenius, for each entry $x$ of its type, $x^2$ divides $\FPdim(\mathcal{C})$. However, regarding Theorem \ref{thm:ModCrit} (3), it is often the case that an entry $x$ in $t_e$ does not satisfy the condition that $x^2$ divides $\FPdim(t'_e) = \FPdim(t_e)/p$, whereas $t'_e$ must be half-Frobenius, therefore, this entry must be split into $p$ entries $x/p, \dots, x/p$ in $t_e'$. This helps to reduce the number of possible $t_e'$.
\end{remark}

The function \texttt{GradingCriteria} in \texttt{TypeCriteria.sage} automates the use of Theorem \ref{thm:ModCrit}, it also iterates over possible types of modularization and rechecks \verb*|Theorem8_7Check| (from Theorem \ref{thm:StrongPrime}) and \texttt{TypeCriteria}.

Here are four examples illustrating the exclusion criteria, each corresponding to a point of Theorem \ref{thm:ModCrit}.
\begin{itemize}
\item[(0)] In the partitioned type $[[1, 2, 2, 3, 3, 3, 3, 3], [1, 2, 2, 3, 6]]$, the neutral component is perfect, but we already know that there is no perfect integral modular fusion category of rank $8$,
\item[(1)] In $[[1, 1, 1, 1, 2, 2, 2, 3, 10, 10], [15], [15], [15]]$, the pointed part is in the neutral component $t_e$, and the entry $2$ has multiplicity three (odd) in $t_e$, but $15$ appears with multiplicity one in some non-neutral components.
\item[(2)] In $[[1, 1, 2, 2, 3, 3, 5, 6, 6, 10, 15], [15, 15]]$, the pointed part is in the neutral component $t_e$, and the entry $5$ has multiplicity one (odd) in $t_e$, but $5$ is not divisible by the prime $2=|B_{pt}|$.
\item[(3)] In $[[1, 1, 2, 12, 15, 15, 20, 20, 20], [30, 30]]$, the pointed part is in the neutral component $t_e$, and the entry $2$ has multiplicity one (odd) in $t_e$. Now, $\FPdim(t_e)/2 = 2^2 3^2 5^2$ which is not a multiple of $20^2$, so the entries $20$ must split for the modularization type $t_e' = [1, 1, 1, 6, 6, 10, 10, 10, 10, 10, 10, 15]$. However, $t_e'$ is excluded by \texttt{TypeCriteria}.
\end{itemize}

\subsection{Congruence Representation} \label{sub:cong}

This subsection reviews some applications of congruence representations of the modular group to modular fusion categories, leading to a proof of the folklore Theorem \ref{thm:folk}. Although a more concise proof is presented later in \S \ref{sub:gal}, the current exposition is meant to be informative and to serve for future research.

As discussed in \cite[Section 3]{NRWW}, a modular fusion category $\mathcal{C}$ is associated with modular data $(S,T)$, which gives a projective representation of $${\rm SL}(2,\mathbb{Z}) = \langle s,t \ | \ (st)^3 = s^2, s^4 = e \rangle.$$ This representation can be lifted to a usual (linear) representation $\rho$ by utilizing the linear characters (i.e. one-dimensional representations), forming a cyclic group of order $12$. This representation is $r$-dimensional—where $r$ represents the rank of $\mathcal{C}$—and is \emph{congruence}. This means it factors through ${\rm SL}(2,\mathbb{Z}/n\mathbb{Z})$, for some $n$ whose smallest one is called the \emph{level}. The level is determined as ${\rm ord}(\rho(t))$ and  satisfies $${\rm ord}(T) \ | \ {\rm ord}(\rho(t)) \ | \  12{\rm ord}(T).$$

A finite-dimensional congruence representation $\rho$ of level $n$ is completely reducible, hence it can be broken down into a direct sum of irreducible representations of ${\rm SL}(2,\mathbb{Z}/n\mathbb{Z})$. It's important to note that this includes only those irreducible representations that do not further factor through ${\rm SL}(2,\mathbb{Z}/d\mathbb{Z})$ for any proper divisor $d$ of $n$. Nevertheless, if $n = \prod_i p_i^{n_i}$ represents the prime factorization of $n$, then $\rho = \bigotimes_i \rho_i$ with each $\rho_i$ being a congruence representation of level $p_i^{n_i}$.

For deeper applications, note that \cite{NWW} proves that the finite-dimensional congruence representations are equivalent to \emph{symmetric} ones, which are classified in \cite{NWW2}.

%[with equality only if $a=1$]
The dimensions $ d $ of the irreducible finite-dimensional congruence representations at level $ n = p^a $ are listed in the table at the end of \cite{NW}.  
For $ a = 1 $, we have $ d \geq (n - 1)/2 $, which implies $ p = n \leq 2d + 1 $.  
For $ a \geq 2 $, we have  
\[
d \geq \frac{n}{2}(1 - \frac{1}{p^2}),
\]
which leads to  
\[
p^a = n \leq 2d + \frac{2d}{p^2 - 1}.
\]
If $ p > 2d + 1 $, then $2d < p - 1$, which implies  
\[
p < p^a < 2d + \frac{1}{p + 1} < 2d + 1,
\]
a contradiction. Therefore, $ p \leq 2d + 1 $ for all $ a $.  
Since the rank $ r $ of the modular fusion category is the sum of the dimensions $ d $ of these irreducible representations, we conclude that $ d \leq r $, and thus $ p \leq 2r + 1 $.

According to Cauchy’s theorem in \cite{BNRW}, the set $S $ of prime factors of $\operatorname{ord}(T) $ coincides with the set of prime factors of the global dimension norm $N $ of the modular fusion category of rank $r $. The prime numbers $p $ mentioned earlier (satisfying $p \leq 2r + 1 $) form the set $S' $, which consists of the prime factors of the level $n $ of the congruence representation.  

Since ${\rm ord}(T) \mid n \mid 12 {\rm ord}(T) $ and $12 = 2^2 \cdot 3 $, it follows that  
$
S \subseteq S' \subseteq S \cup \{2,3\}.
$
Therefore, for all prime factors $p \neq 2,3 $ of $N $, we have $p \leq 2r + 1 $. This inequality holds trivially for $p = 2,3 $.  

This concludes our first proof of Theorem \ref{thm:folk} without relying on \cite[Theorem II (iii)]{DLN}. 

\subsection{Galois Action} \label{sub:gal}

Let $(s,t)$ denote a \emph{normalized} modular data, and let $\mathbb{Q}_n$ be the cyclotomic field $\mathbb{Q}(\zeta_n)$, where $n = \text{ord}(t)$. In this subsection, $\sigma$ will refer to a Galois automorphism in $\text{Gal}(\mathbb{Q}_n/\mathbb{Q})$. For simplicity, we use the same symbol $\sigma$ to denote the induced permutation $X \mapsto \sigma(X)$ on the simple objects. It acts on $\dim$, $s$, and $t$ as follows: 
\begin{itemize}
\item[(1)] $\sigma(\dim(X)^2) = \frac{\sigma(\dim(\mathcal{C}))}{\dim(\mathcal{C})}\dim(\sigma(X))^2$, see \cite{EGNO},
\item[(2)] $\sigma(s_{X,Y}^2) =  s_{X,\sigma(Y)}^2$, see \cite{EGNO},
\item[(3)] $\sigma^2(t_X) = t_{\sigma(X)}$, see \cite[Theorem II (iii)]{DLN}.  
\end{itemize} 
See for example \cite[Section 2]{PSYZ} for an explicit normalization of the modular data. The following result is a straightforward consequence of (3) and \cite{BNRW}.

\begin{theorem} \label{thm:folk}
For any prime factor $p$ of the dimension norm of a modular fusion category with rank $r$, it holds that $p \leq 2r + 1$.
\end{theorem}
\begin{proof}
If $p=2,3$ then $p \le 2r+1$ trivially as $r \ge 1$. Let $p \neq 2,3$ be a prime factor of the global dimension norm. By Cauchy theorem in \cite{BNRW}, $p$ divides ${\rm ord(t)}$. So there must be a simple object $X$ such that $p$ divides the order of $t_X$, thus  the orbit of $(\sigma^2(t_X))$ has at least $(p-1)/2$ distinct elements, because the group of units in $\mathbb{Z}/p\mathbb{Z}$ is cyclic of order $p-1$, so it has an element $g$ with ${\rm ord}(g^2) = (p-1)/2$. So by (3), $r \ge (p-1)/2$, i.e., $p \le 2r+1$. 
\end{proof}

Here is an example of type of rank $r=13$ and $\FPdim = 2^4 3^2 5^2 7^2 19^2 37^2$ excluded Theorem \ref{thm:folk}:
$$[1, 777, 1036, 1295, 3990, 4218, 24605, 42180, 98420, 98420, 147630, 147630, 147630],$$
because $p=37 > 2r + 1 = 27$. 

Here is a stronger version in the integral case (shared by Eric Rowell and Andrew Schopieray):

\begin{theorem}  \label{thm:RoSc} For an integral modular fusion category, for every prime $p$ dividing the global ${\rm FPdim}$, there is a basic ${\rm FPdim}$ of multiplicity $m$ such that $p \le 2m+1$.
\end{theorem} 
 \begin{proof}
Consider the orbit $(\sigma^2(t_X))$ with at least $(p-1)/2$ distinct elements from above proof of Theorem \ref{thm:folk}. By (3), the orbit $(\sigma(X))$ has also at least $(p-1)/2$ distinct elements. By applying (1) on the (weakly) integral case, we get that $\sigma({\rm FPdim}(X)) = {\rm FPdim}(X)$. Thus all simple objects in the orbit $(\sigma(X))$ has the same ${\rm FPdim}$, so the multiplicity $m$ of this basic ${\rm FPdim}$ satisfies $m \ge (p-1)/2$, i.e., $p \le 2m+1$.
 \end{proof}

Here is an example of type of rank $13$ and $\FPdim = 2^4 3^6 5^2 7^2 17^2$ excluded by Theorem \ref{thm:RoSc} (but not Theorem \ref{thm:folk}):
$$[1, 238, 459, 540, 595, 918, 5355, 9180, 21420, 21420, 32130, 32130, 32130],$$
because $p=17 > 2m+1 = 7$, where $m=3$ is the largest multiplicty of a basic $\FPdim$.

\begin{remark}
As mentioned by an anonymous referee, a refined version of Theorem \ref{thm:RoSc} follows from \cite[Lemma 4.6]{BGNPRW}: in a modular fusion category, for every odd prime $p$ dividing $\ord(t)$, there exists a basic $\dim$ with Galois multiplicity (that is, the number of Galois-conjugate basic $\dims$) $m \ge \varphi(p^a)/2 \ge (p-1)/2$, where $a = v_p(\ord(t))$.
\end{remark}

%\begin{remark}
%The assertion of Theorem \ref{thm:RoSc} can be generalized beyond the integral assumption, albeit with the necessity to substitute the usual multiplicity with the Galois-multiplicity, that is, the count of Galois-conjugate basic $\dim$ of the given basic $\dim$.
%\end{remark}

Here is a stronger version of Theorem \ref{thm:RoSc}:  
\begin{theorem} \label{thm:StrongPrime}
For an integral modular fusion category, let $S$ be the set of odd prime factors of the global ${\rm FPdim}$. There is a partition $(S_i)$ of $S$, and multiplicities $(m_i)$ of \emph{some} distinct basic $\FPdims$ such that $$m_i \ge \frac{1}{2} \lcm_{p \in S_i}(p-1).$$  
\end{theorem} 
 \begin{proof}  
Let $S_X$ be the set of odd prime divisors of $\text{ord}(t_X)$. By Cauchy theorem in \cite{BNRW}, the union of  the sets $S_X$ over all the simple objects $X$ is exactly $S$. Let $\lambda$ be the \emph{Carmichael function}, i.e. the exponent of the multiplicative group of integers modulo $n$. It is well-known that if $n = \prod_i p_i^{n_i}$ is the prime factorization of $n$, then $\lambda(n) = \lcm_i \lambda(p_i^{n_i})$, whereas for $p_i$ odd then $\lambda(p_i^{n_i}) = \varphi(p_i^{n_i}) = (p_i-1)p_i^{n_i-1}$, where $\varphi$ is the Euler totient function. The orbit $(\sigma(X))$ has at least $\frac{1}{2}\lambda(\prod_{p \in S_X} p)$ distinct elements. Therefore, by above and the proof of Theorem \ref{thm:RoSc}, the multiplicity of $\FPdim(X)$ is at least $\frac{1}{2} \lcm_{p \in S_X}(p-1)$. The result follows.
 \end{proof}
 
Here is an example of type of rank $25$ and $\FPdim = 3^4 5^2 7^2 11^2 13^2$ excluded by Theorem \ref{thm:StrongPrime} (but not Theorem \ref{thm:RoSc}):
\[
[[1,1],[39,2],[231,2],[273,2],[1001,2],[1287,2],[3465,2],[4095,2],[9009,2],[15015,8]].
\]
First, note that it is not excluded by Theorem~\ref{thm:RoSc}, since 
$m = 8 \ge (p-1)/2$ for all $p \in \{3,5,7,11,13\}$. 
Next, let us explain why it is excluded by Theorem~\ref{thm:StrongPrime}. 
Let $S_i$ and $S_j$ be the components of the partition containing $11$ and $13$, respectively. 
The multiplicity $m = 8$ is the only one larger than $(p-1)/2$ for $p = 11, 13$. 
Hence $S_i = S_j$, but 
\(
\lcm(13-1, 11-1)/2 = 30 > 8,
\)
a contradiction.

The checking of Theorems \ref{thm:RoSc} and \ref{thm:StrongPrime} are automated by the function \verb*|Theorem8_5Check| and \verb*|Theorem8_7Check| in \texttt{TypeCriteria.sage} in \cite{data}.

\section{Stronger Arithmetic Constraints} \label{sec:SAC}

In this section, we will first prove some arithmetic constraints on the rank of the Drinfeld center of $\Rep(G)$, for any finite group $G$. Then, we will discuss how far they can be generalized to any integral modular fusion categories.

\subsection{Rank of $\mathcal{Z}(\Rep(G))$} \label{sub:ZG}

This subsection is inspired from discussions with Geoff Robinson and Dave Benson in \cite{PalMO, PalMO2}. The goal is to simplify the group theoretic way to express the rank of $\mathcal{Z}(\Rep(G))$, for any finite group $G$, and to provide some bounds involving the prime divisors of $|G|$. Consequently, this class of integral modular fusion categories satisfies Conjectures \ref{conj:pler} and \ref{conj:pler2}. We will adopt the following notations throughout our discussion:
\begin{itemize}
\item $Z(G)$ denotes the center of $G$,
\item $\Gamma_G$ represents a complete set of conjugacy class representatives,
\item $c_G$ is the rank of $\Rep(G)$, which equals the total number of conjugacy classes in $ G $, denoted by $ |\Gamma_G| $,
\item $r_G$ refers to the rank of $\mathcal{Z}(\operatorname{Rep}(G))$. 
\end{itemize}
%According to Lemma \ref{lem:pair}, this is equivalent to the number of conjugacy classes of pairs of commuting elements.
According to \cite{CGR} or \cite[Section 3]{NN}, the rank \( r_G \) is determined by the number of irreducible characters in the centralizers of class representatives of $G$, as expressed precisely in Equation (\ref{equ:r1}) below.
\begin{lemma} \label{lem:rank}
Let $G$ be a finite group. Then 
\begin{align}
\label{equ:r1} r_G = \sum_{a \in \Gamma_G} c_{C_G(a)}.
\end{align}
\end{lemma}

The following result provides a simpler expression for the rank $r_G$. An anonymous referee pointed out that it is already stated in the expository paper \cite[Corollary 2.3]{Ma95}.

\begin{lemma} \label{lem:pair}
The rank $r_G$ is the number of conjugacy classes of ordered pairs of commuting elements of $G$.%, i.e. the cardinality of the 
\end{lemma}
\begin{proof}
By the equality in Lemma \ref{lem:rank}, it suffices to establish a bijection between the set $$ A_G := \{ c(a_1,a_2) \ | \ a_1,a_2 \in G \text{ with } a_1a_2 = a_2a_1\},$$
where $$c(a_1,a_2) := \{ (ga_1g^{-1},ga_2g^{-1}) \ | \ g \in G \},$$ and the set $$ B_G := \{(a, \beta) \ | \ a \in \Gamma_G \text{ and } \beta \text{ is a conjugacy class within } C_G(a)\}.$$
Given $c(a_1,a_2) \in A_G$, we associate the element $(a_1,\{ ha_2h^{-1} \ | \ h \in C_G(a_1)\}) \in B_G$. We merely need to confirm that if $a_1 = ga_1g^{-1}$, then $a_2$ and $ga_2g^{-1}$ are conjugates in $C_G(a_1)$, which is apparent since $a_1 = ga_1g^{-1}$ means that $g \in C_G(a_1)$.  
Given $(a,\beta) \in B_G$, we associate the element $c(a,b) \in A_G$, where $b \in \beta$. We only need to verify that if $b' \in \beta$, then $c(a,b') = c(a,b)$. Note that $b' = hbh^{-1}$, where $h \in C_G(a)$. Therefore, $c(a,b) = c(hah^{-1},hbh^{-1}) = c(a,b')$ because $hah^{-1} = a$, given that $h \in C_G(a)$.
\end{proof}

The following Proposition \ref{prop:conjcheck}, Theorem \ref{thm:r/2}, and Theorem \ref{thm:r/3} form a sequence of increasingly stronger results, each proven independently. We include all the proofs, rather than only the final one (see \S\ref{sec:geoff}), due to the significant increase in complexity. Recall that a prime $p$ divides $\FPdim(\mathcal{Z}(\operatorname{Rep}(G))) = |G|^2$ if and only if it divides $|G|$.
\begin{proposition} \label{prop:conjcheck}
Let $ G $ be a finite group. Let $ p $ be a prime divisor of $ |G| $. Then $ r_G \geq p $. 
\end{proposition}
\begin{proof}
This proof is due to Dave Benson. The ordered pair of commuting elements $(g,g^i)$ for $1 \le i \le {\rm ord}(g)$ are all in distinct conjugacy classes, so by Lemma \ref{lem:pair}, $r_G \ge {\rm ord}(g)$ for all $g$ in $G$. By Cauchy's theorem, there is an element of order $p$ dividing $|G|$. Thus, $r_G \ge p$.
%By (\ref{equ:r3}), $r_G \ge \sum_{g \in \Gamma_{G}} {\rm ord(g)}$, in particular, for all $g$ in $G$ then $r_G \ge {\rm ord(g)}$. By Cauchy's theorem, for all prime $p$ dividing $|G|$, there is $g$ in $G$ such that ${\rm ord(g)} = p$, therefore $r_G \ge p$. 
\end{proof} 
The number of conjugacy classes of ordered pairs of commuting elements in the alternating group $A_n$ is $1, 1, 9, 14, 22, 44, 74$ for $n=1,\dots, 7$, respectively, see \cite{A371059}.

\begin{lemma} \label{lem:conjcheck2}
Let $G$ be a finite group. Then 
\begin{align}
\label{equ:r2} r_G & \geq |Z(G)|c_G + \sum_{g \in \Gamma_{G} \backslash Z(G) } {\rm ord(g)} , \\ 
\label{equ:r3} r_G & \geq \sum_{g \in \Gamma_{G}} {\rm ord(g)}.
\end{align}
\end{lemma}
\begin{proof}
In general, we have 
\begin{align} \label{equ:r4} 
c_{C_{G}(a)} \geq |Z(C_{G}(a))|  \geq {\rm ord}(a)
\end{align}
but if $a \in Z(G)$ then $C_G(a) = G$ and so $c_{C_G(a)} =c_G $. The inequalities (\ref{equ:r2}) and (\ref{equ:r3}) follow from (\ref{equ:r1}).
\end{proof}

\begin{theorem} \label{thm:r/2} 
Let $ G $ be a finite group. Let $ p $ be a prime divisor of $ |G| $. Then $ r_G \geq 2p $.
\end{theorem}
\begin{proof}
By Cauchy's theorem, there is an element $g$ of order $p$. If $p=2$ then $r_G \ge 3$ by (\ref{equ:r3}). So if $r_G \le 3$, then $r_G = 3$ and $|\Gamma_G| = 2$ by (\ref{equ:r3}), thus $G=C_2$, contradiction. So $r_G \ge 4 = 2p$. 

Therefore, we can assume that $p$ is odd. If $g^2$ is not in the conjugacy class of $g$ then, by (\ref{equ:r1}) and (\ref{equ:r4}), $$r_G \geq c_{C_{G}(g)} + c_{C_{G}(g^2)} \geq |Z(C_{G}(g)| + |Z(C_{G}(g^2)| \geq 2p,$$ because $x \in Z(C_{G}(x))$ and $\ord(g^2) = p$, as $p$ is odd. Thus, we can assume the existence of $h$ in $G$ such that $hgh^{-1} = g^2$, but then $h^ngh^{-n} = g^{2^n}$. Fermat's little theorem states that $2^{p-1} \equiv 1 \mod p$, and the multiplicative group $(\mathbb{Z}/p\mathbb{Z})^{\times}$ is cyclic of order $p-1$. So $h^{n}gh^{-n} = g$ for $n=p-1$, and $\{ h^ngh^{-n} \ | \ n=1,\dots,p-1\} = \langle g \rangle \setminus \{e\}$. Thus $p-1$ divides $\ord(h)$. But $g$ and $h$ have different order, so cannot be in the same conjugacy class, so by (\ref{equ:r3}), $r_G \ge \ord(e) + \ord(g) + \ord(h) \ge 1+p + p-1 = 2p$.
\end{proof}

The equality $r_G=2p$ is realized by $(G,p)=(C_2,2)$. Out of this example, we can get even better:

\begin{theorem} \label{thm:r/3} 
Let $ G $ be a finite group. Let $ p $ be a prime divisor of $ |G| $. Then $ r_G \geq 3p - 1 - \delta_{G, C_2} $.
\end{theorem}
\begin{proof}
If $ G $ is Abelian, then by (\ref{equ:r2}), we have $ r_G = |G|^2 $. This means that $ r_G < 3p - 1 $ implies $ G = C_2 $, where $ r_G = 4 = 3p - 2 $. The non-Abelian case, attributed to Geoff Robinson, is much more complex and too lengthy to include here. It is presented in \S\ref{sec:geoff}.
\end{proof}

The equality $r_G=3p-1$ is realized by $(G,p)=(S_3,3)$. We can expect even better in the non-solvable case:

\begin{conjecture} \label{conj:r/5}
For every non-solvable finite group $G$, and for every prime $p$ dividing the order of $G$, the inequality $r_G \geq 5p - 3$ holds.
\end{conjecture}

This has been verified for all non-solvable groups of order less than $1920$ and for all non-Abelian finite simple groups of order less than $10^8$. The equality is achieved for $G=A_5$ and ${\rm PSL}(2,7)$. Geoff Robinson has shown in \cite{PalMO2} that $r \ge 5p - 1$ for any non-solvable groups without a self-centralizing Sylow $p$-subgroup of order $p$, where $p$ the largest prime factor of $|G|$.
%must satisfy a specific condition related to the index of a centralizer in a normalizer.

\subsection{Integral Modular Fusion Categories} \label{sub:ArCo}
Referring to the notations used previously, recall that Theorem \ref{thm:folk} asserts that \(p \leq 2r + 1\), which is enhanced for the integral case by Theorem \(\ref{thm:RoSc}\) to \(p \leq 2m + 1\), and also consult Theorem \ref{thm:StrongPrime} for a more robust version. Now, let's explore the extent to which the arithmetic constraints discussed in \S \ref{sub:ZG} can be applied to integral modular fusion categories. The subsequent conjecture aims to expand upon Proposition \ref{prop:conjcheck}.
  
\begin{conjecture} \label{conj:pler}
For any prime number $p$ that divides the global dimension of a integral modular fusion category with rank $r$, then $p \leq r$.
 \end{conjecture}
 
\begin{proposition} \label{prop:ConjR21}
The statement of Conjecture \ref{conj:pler} is true up to rank $21$.
\end{proposition} 
\begin{proof}
The proof largely relies on computer assistance. All computational steps are documented in \texttt{InvestUpToRank21par.txt} in \cite{data}. The script \texttt{all\_rep\_above\_r}, mentioned in \S \ref{sec:Egy}, focuses on types that comply with Theorem \ref{thm:StrongPrime} yet contradict Conjecture \ref{conj:pler}. Consequently, we identified exactly $187$ potential types up to rank $21$, as listed in \texttt{UpToRank21par.txt} in \cite{data}. Specifically, we found none up to rank $13$, and then $1, 4, 22, 2, 28, 0, 8, 122$ types at ranks $14$ through $21$, respectively. Of these, only $40$ types pass \texttt{TypeCriteria} from \S \ref{sec:crit}. Among these, only $28$ types pass \texttt{GradingCriteria} from \S \ref{sub:univ}. They are ultimately excluded by our fusion ring solver's partition version in \S \ref{sub:DimPar}, limited to one second per type. 
\end{proof}

The extension of Proposition \ref{prop:ConjR21} is currently underway, as detailed in \texttt{InvestAboveRank21par.txt}. At ranks $22$, $23$, and $24$, there are still $22$, $1$, and $16$ types respectively remaining to be considered; all of them being perfect, thus:

\begin{proposition} \label{prop:ConjR24Perf}
The statement of Conjecture \ref{conj:pler} is true up to rank $24$ in the non-perfect case.
\end{proposition} 

If Conjecture \ref{conj:pler} is true then it is optimal as demonstrated by the pointed examples of prime rank. Thus, we could expect better for the non-pointed case. The following conjecture generalizes Theorem \ref{thm:r/2}.

\begin{conjecture} \label{conj:pler2}
For any prime number $p$ that divides the global dimension of a non-pointed integral modular fusion category with rank $r$, then $p \leq r/2$.
\end{conjecture}

\begin{proposition} \label{prop:ConjR15}
The statement of Conjecture \ref{conj:pler2} is true up to rank $15$.
\end{proposition} 
\begin{proof}
The proof is mainly computer-assisted, but computationally a bit harder than for Proposition \ref{prop:ConjR21}. All computational steps are documented in \texttt{InvestUpToRank15pahr.txt} in \cite{data}. The specialized script \texttt{all\_rep\_above\_half\_r}, mentioned in \S \ref{sec:Egy}, provides $3094$ non-pointed types up to rank $15$ satisfying Theorem \ref{thm:StrongPrime} but contradicting Conjecture \ref{conj:pler2}, available in \texttt{UpToRank15pahr.txt}. We found none up to rank $8$, and then $6, 36, 250, 2266, 45, 491$ types at ranks $9$ through $15$, respectively. The irregularity between rank $13$ and rank $14$ comes from the fact that if $p>14/2 = 7$, then $p\ge 11$ because $9$ is not prime. Of these, only $1256$ types pass \texttt{TypeCriteria} from \S \ref{sec:crit}. Among these, only $900$ types pass \texttt{GradingCriteria} from \S \ref{sub:univ}. The use of our fusion ring solver's partition version from \S \ref{sub:DimPar}, limited to one second per type, reduced the rest to just $15$ types, then $12$, $5$, $4$ ones with $10$, $100$, $1000$ seconds per type respectively. Next, we applied our full fusion ring solver's partition version from \S \ref{sub:full}, requiring to specify the duality, so our $4$ types becames $40$ cases, we simplify with the commutativity option. After the first round, limited to $10$ seconds per case, there remain $8$ cases from $2$ types, ultimately all excluded in about two hours.
%below:
%\begin{itemize}
%\item $[1,70,75,150,168,175,350,525,1400,1400,2100,2100,2100],$
%\item $[1,105,175,182,300,390,3900,6825,9100,9100,9100,9100,9100,9100,13650],$
%\end{itemize}
\end{proof}
 
The extension of Proposition \ref{prop:ConjR15} is currently in progress, as detailed in \texttt{InvestAboveRank15pahr.txt}. At ranks $16$ and $17$, there remain $132$ types (all perfect) and $7737$ types (of which only $19$ are not perfect), respectively, still to be considered. Thus:

\begin{proposition} \label{prop:ConjR16Perf}
The statement of Conjecture \ref{conj:pler} is true up to rank $16$ in the non-perfect case.
\end{proposition}  

With our current knowledge, it would not be reasonable to conjecture a generalization of Theorem \ref{thm:r/3} suggesting that $p \le (r+1)/3$ in the non-pointed case, or of Conjecture \ref{conj:r/5} suggesting that $p \le (r+3)/5$ in the non-solvable case. But it makes sense to explore in this direction. 

\begin{question} \label{q:NonSolvIntMFC}
Is it true that for any prime number $p$ dividing the global dimension of a non-solvable integral modular fusion category with rank $r$, the inequality $p \leq (r+3)/5$ holds?
\end{question}
The following proposition is related to Question \ref{Q:RankLess22}.
\begin{proposition} \label{prop:qr22}
An affirmative response to Question \ref{q:NonSolvIntMFC} would indicate that the minimum rank required for a non-trivial perfect integral modular fusion category is $22$.
\end{proposition}
\begin{proof}
By \cite[Proposition 4.5 (iv)]{ENO11}, a non-trivial perfect integral fusion category $\mathcal{C}$ is non-solvable. Thus by \cite[Theorem 1.6]{ENO11}, $\FPdim(\mathcal{C})$ must have at least three distinct prime factors, so there must be a prime factor $p \ge 5$. Assume that $\mathcal{C}$ is modular of rank $r$, then an affirmative response to Question \ref{q:NonSolvIntMFC} implies that $r \ge 5p-3 \ge 22$. Finally, the bound is realized by $\mathcal{Z}({\rm Rep}(A_5))$ having rank $22$.
\end{proof}
 
Here is what we can deduce in the perfect integral case:

\begin{corollary} \label{cor:pler}
For any prime number $p$ that divides the global dimension of a non-trivial perfect integral modular fusion category with rank $r$, then $p \leq 2r-5$.
\end{corollary}
\begin{proof}
By Theorem \ref{thm:small}, the number of distinct basic $\FPdim$s is at least $4$. Let $p$ be the biggest prime divisor of the global dimension, then by Theorem \ref{thm:RoSc}, there is a basic $\FPdim$ of multiplicity $m \ge (p-1)/2$. Thus $r \ge m+3 = (p+5)/2$. The result follows.
\end{proof}

\section{Proof Up To Rank $13$} \label{sec:R13}
This section proves Theorem \ref{thm:main} by leveraging the advanced results from Sections \ref{sec:AdvMD} and \ref{sec:SAC}, complemented by computational assistance. All computational steps are documented in \texttt{InvestUpToRank13.txt}, as referenced in \cite{data}. The initial step involves classifying all the possible types (4308 in total) constrained by Theorem \ref{thm:StrongPrime}, using \texttt{all\_rep\_th87} from \S \ref{sec:Egy}. Applying Proposition \ref{prop:ConjR15} reduces this number to just 1722 types. From these, merely $794$ types met the \texttt{TypeCriteria} outlined in \S \ref{sec:crit}. Among these, only $439$ types fulfilled the \texttt{GradingCriteria} presented in \S \ref{sub:univ}. Utilizing the partition version of our fusion ring solver, as discussed in \S \ref{sub:DimPar} and limiting the processing time to one second per type, further reduced the count to just $38$ types. Subsequently, excluding the type $[1,1,5,5,6,6,20,24,30,40,60,60,60]$ in under $1000$ seconds brought the total down to $37$ types. 

The distribution of these types across ranks $7$ to $13$ is as follows: $1, 1, 2, 2, 4, 7, 20$, respectively. Out of these, $27$ types are non-perfect, and $10$ are perfect, with each perfect type containing at least one prime-power entry, thereby providing direct proof for the simple case within Theorem \ref{thm:perfectcat13}. Following this, we employed the full version of our fusion ring solver, as outlined in \S \ref{sub:full}, which necessitated specifying duality. This adjustment resulted in the expansion of our $37$ types into $572$ cases. We streamlined these cases using the commutativity and \verb*|FusionData| options, as elucidated in \S \ref{sub:user}. After an initial round of processing, limited to $10$ seconds per case, $56$ cases remained unresolved, plus $81$ cases with solutions. Further rounds of processing, with limits set to $100$ seconds and then $1000$ seconds, resulted in $29+101$ and $23+106$ cases, respectively. Ultimately, leveraging HPC capabilities allowed for the resolution of all cases, concluding with solutions for $120$ cases across $26$ distinct types, as listed below:
 
\begin{itemize} {\small
 \item Rank $7$: [1, 1, 1, 1, 2, 2, 2],
 \item Rank $8$: [1, 1, 2, 2, 2, 2, 3, 3],
 \item Rank $9$: [1, 1, 1, 1, 4, 4, 6, 6, 6], [1, 1, 2, 2, 2, 2, 3, 3, 6],
 \item Rank $10$: [1, 1, 1, 2, 2, 2, 2, 2, 2, 3], [1, 1, 2, 3, 3, 4, 4, 4, 6, 6],
 \item Rank $11$: [1, 1, 1, 1, 2, 2, 2, 2, 2, 2, 2],
  [1, 1, 1, 1, 2, 4, 4, 4, 4, 6, 6],
  [1, 1, 1, 1, 4, 4, 12, 12, 18, 18, 18],
  [1, 1, 1, 3, 4, 4, 4, 4, 4, 4, 6],
 \item Rank $12$: [1, 1, 1, 1, 3, 3, 3, 3, 4, 4, 6, 6],
  [1, 1, 2, 2, 2, 2, 3, 3, 3, 3, 3, 3],
  [1, 1, 2, 2, 2, 2, 3, 3, 6, 6, 6, 6],
  [1, 1, 2, 2, 2, 2, 6, 6, 6, 6, 9, 9],
  [1, 1, 2, 3, 3, 6, 6, 8, 8, 8, 12, 12],
  [1, 1, 2, 6, 6, 6, 6, 10, 10, 10, 15, 15],
 \item Rank $13$: [1, 1, 1, 1, 2, 2, 4, 4, 4, 4, 4, 4, 6],
  [1, 1, 1, 1, 4, 4, 4, 4, 4, 4, 10, 10, 10],
  [1, 1, 1, 1, 4, 4, 6, 6, 6, 6, 6, 6, 6],
  [1, 1, 1, 1, 4, 4, 6, 12, 12, 12, 12, 18, 18],
  [1, 1, 1, 1, 4, 4, 12, 12, 36, 36, 54, 54, 54],
  [1, 1, 1, 3, 12, 12, 20, 20, 20, 20, 20, 20, 30],
  [1, 1, 2, 2, 2, 2, 3, 3, 3, 3, 3, 3, 6],
  [1, 1, 2, 2, 2, 2, 3, 3, 6, 6, 6, 6, 6],
  [1, 1, 2, 2, 2, 2, 6, 6, 6, 6, 9, 9, 18],
  [1, 1, 2, 3, 3, 24, 30, 30, 40, 40, 40, 60, 60].}
\end{itemize} 
  
Note that none are perfect, which concludes the proof of Theorem \ref{thm:perfectcat13} at this stage. The total number of commutative fusion rings obtained is $30084$. Subsequently, the function \texttt{ModularGrading} in \texttt{ModularData.sage} narrows this down to $8720$ fusion rings that have a modular grading (refer to Definition \ref{def:ModPart}), across $14$ types (a new feature in Normaliz allows for the direct classification of fusion rings with modular grading; see \cite[\S H.5.2]{NorManual}). However, only $42$ fusion rings (spanning $10$ types) are both cyclotomic and self-transposable (refer to \S \ref{sub:Smat}), distributed as $4, 0, 6, 6, 16, 10$ across ranks $8$ to $13$. These are detailed in \texttt{FusionRingsCCSTMG.sage}. The subsequent use of \texttt{MagicCriterion} from \S \ref{sub:Tmat} further refines the list to $13$ fusion rings (across $5$ types), distributed as $4, 0, 2, 4, 6, 0$. Notably, this stage confirms the absence of any non-pointed integral modular fusion categories of rank $13$. Lastly, applying \texttt{STmatrix} narrows it down to $3$ (non-pointed) types encompassing $5$ fusion rings and $19$ modular data (MD):
\begin{itemize}
\item $[1, 1, 2, 2, 2, 2, 3, 3]$ with $8$ MD,
\item $[1, 1, 1, 2, 2, 2, 2, 2, 2, 3]$ with $3$ MD,
\item $[1, 1, 1, 1, 2, 2, 2, 2, 2, 2, 2]$ with $8$ MD,
\end{itemize}
The pointed cases were addressed by using \texttt{STmatrixCo} on a list of fusion rings for the abelian groups of order up to $13$. We concluded with $19+64$ sets of modular data, originating from $5+18$ fusion rings representing $3+13$ types (non-pointed + pointed), up to rank $13$. This proves Theorem \ref{thm:main}. The full details are available in \cite[\S\ref{sec:modata}]{ABPPsupp} or \texttt{MDuptoRank13.txt} at \cite{data}.

\begin{remark}
When considering isomorphism classes, it is appropriate to adopt a \emph{normal form} by sorting the basic elements according to their $\FPdim$ \emph{and} topological spin, and limit basis permutations to those preserving both.
\end{remark}

The \verb*|STmatrix| function ran quickly on all remaining fusion rings, with the exception of two of the following type which took several hours to be excluded: $[1, 1, 2, 2, 2, 2, 6, 6, 6, 6, 9, 9, 18].$ 

\section{About Ranks $14$ and $15$} \label{sec:r14r15}
The current state of exploration at ranks $14$ and $15$ is documented in \texttt{InvestRank14.txt} and \texttt{InvestRank15.txt}, in \cite{data}.  Regarding rank $14$, the initial count was $29113$ (non-pointed) types, constrained by Theorem \ref{thm:StrongPrime}. At the time of writing, only $35$ types remain to be considered, $8$ of which are non-perfect. This leads to the following propositions:

\begin{proposition} \label{prop:NPrank14}
The type of a non-perfect non-pointed integral modular fusion category of rank $14$ belongs to the following list: 
\begin{itemize}
\item $[1, 1, 2, 3, 3, 24, 24, 42, 42, 56, 56, 56, 84, 84]$,
\item $[1, 1, 2, 3, 3, 24, 120, 150, 150, 200, 200, 200, 300, 300]$,
\item $[1, 1, 24, 24, 36, 40, 45, 45, 90, 90, 90, 180, 180, 180]$,
\item $[1, 1, 40, 84, 90, 126, 315, 315, 504, 630, 840, 1260, 1260, 1260]$,
\item $[1, 1, 45, 45, 90, 140, 168, 168, 630, 630, 840, 1260, 1260, 1260]$,
\item $[1, 1, 60, 60, 84, 140, 189, 189, 540, 1260, 1260, 1890, 1890, 1890]$,
\item $[1, 1, 90, 90, 90, 108, 140, 378, 945, 945, 1260, 1890, 1890, 1890]$,
\end{itemize}
together with 
\begin{itemize}
\item $[1, 1, 1, 3, 3, 3, 3, 3, 4, 4, 4, 4, 4, 4]$ 
\end{itemize}
which admits four commutative cyclotomic self-transposable fusion rings, leading to $6$ MD, all from $\mathcal{Z}(\VVec(A_4,\omega))$, already listed in the database of \cite{GrMo} at directory path \verb*|/Modular_Data/12/3/|. 
\end{proposition}
For the perfect case, you can consult the list of $27$ possible types in \texttt{InvestRank14.txt}.
\begin{corollary}
The set of prime divisors of the global $\FPdim$ of a non-pointed integral modular fusion category of rank $14$ must be:
\begin{itemize}
\item $\{2, 3, 5, 7\}$ in the perfect case,
\item $ \{2, 3, 5\}, \{2, 3, 5, 7\}$, or $\{2, 3, 7\}$ in the non-perfect case, except $\mathcal{Z}(\VVec(A_4,\omega))$ for which it is $\{2, 3\}$.
\end{itemize}
\end{corollary}

Regarding rank $15$, the initial count was $333423$ (non-pointed) types, also constrained by Theorem \ref{thm:StrongPrime}. At the time of writing, only $9027$ types remain to be considered, of which $399$ are non-perfect. The list of possible types are in \texttt{InvestRank15.txt}.
\begin{corollary}
The set of prime divisors of the global $\FPdim$ of a non-pointed integral modular fusion category of rank $15$ must be:
\begin{itemize}
\item $\{2, 3, 5\}, \{2, 3, 5, 7\}$, or $\{2, 3, 7\}$ in the perfect case,
\item $\{2, 3\}, \{2, 3, 5\}, \{2, 3, 5, 7\}$, or $\{2, 3, 7\}$ in the non-perfect case.
\end{itemize}
\end{corollary}

\begin{remark} \label{ENOcrit}
In the subsequent paper \cite{DP25}, a generalization of an argument of Etingof--Nikshych--Ostrik yields a powerful necessary criterion for integral modular categorification. Using this criterion, we complete the classification of categorifiable integral modular data up to rank~$14$, and up to rank~$25$ in the odd-dimensional case. We also demonstrate the ongoing effectiveness of this approach by reducing the number of unresolved types at rank~$15$ from $9027$ to just $2481$.
\end{remark}

\section{The Odd-Dimensional Case}
\label{sec:MNSD}

For an overview of the current state of knowledge on odd-dimensional modular fusion categories, we refer the reader to \cite{CzPl,CzPl2}. A foundational result in this area establishes that an odd-dimensional modular fusion category $\mathcal{C}$ is equivalent to being maximally non self-dual (MNSD), meaning that its only self-dual simple object is the unit object. Let $(d_i)_{i \in I}$ represent the basic $\FPdim$s of $\mathcal{C}$. Since $d_i^2$ is a divisor of the odd $\FPdim(\mathcal{C})$, each $d_i$ must be odd. Furthermore, the equation $\sum_{i \in I} d_i^2 = \FPdim(\mathcal{C})$ implies that the rank $r=|I|$ must also be odd. This reduces our investigation to Egyptian fractions of the form $q = \sum_{i = 1}^r \frac{1}{s_i^2}$, where $q, r, s_i \in \mathbb{Z}_{\ge 1}$, $s_1\ge \cdots \ge s_r \ge 1$, and both $r$ and $s_i$ are odd. Additionally, $s_i$ divides $s_1$ for all $i$, and $s_{2k} = s_{2k+1}$. This yields the expression
$$ q = \frac{1}{s_1^2} + \sum_{k=1}^{(r-1)/2}  \frac{2}{s_{2k}^2}. $$
Since each $s_i$ is odd, we have $s_i^2 \equiv 1 \pmod{8}$, which implies $q \equiv r \pmod{8}$ and that $q$ is odd as well. Utilizing a similar technique as in \S \ref{sec:Egy}, we can assume $s_i > 1$ (hence $s_i \ge 3$), by completing the classification with additional $1$s if necessary. Consequently, we can assume $q \le r/9$. For $r < 27$, this allows us to deduce that $q=1$, and therefore $r \equiv 1 \pmod{8}$, which narrows the possibilities for $r$ to $1, 9, 17, 25$ (up to completing by $1$s).

\begin{remark}
This strategy can be extended. For instance, by adding eighteen $3$s to complete the classification, we may assume that $s_i \ge 5$ for $i \le r-16$, which leads to $q \le 16/9 + (r-16)/25$. If $r < 47$ (which becomes $51$ because $q \equiv r \pmod{8}$), we can assume that $q=1$. However, this extended strategy will not be applied in this paper.
\end{remark}

For all $r<25$, we compiled the following list of all possible non-pointed types (using \texttt{all\_rep\_MNSD} from \S \ref{sec:Egy}):
\begin{itemize}
\item[•] [[1, 9], [3, 8], [9, 2a]],
\item[•] [[1, 7], [3, 2], [5, 8], [15, 2a]],
\item[•] [[1, 3], [3, 8], [5, 6], [15, 2a]],
\item[•] [[1, 1], [3, 2], [7, 2], [9, 4], [21, 8], [63, 2a]],
\item[•] [[1, 1], [9, 4], [25, 2], [45, 2], [75, 8], [225, 2a]],
\end{itemize}
where $a \ge 0$ represents the number of $1$s added for completion. It is noteworthy that these ranks are $17+2a$, which corroborates a result from \cite{CzPl} stating that any odd-dimensional modular fusion category with rank less than $17$ is pointed. Further, \cite[Remark 4.3]{CzPl} states that any perfect odd-dimensional modular fusion category is a Deligne product of simple ones. From the preceding analysis, a non-pointed one must have a rank of at least $17$, meaning a perfect non-simple one must have a rank of at least $289$ ($=17^2$). Therefore, a perfect one with a rank less than $289$ must be simple, and cannot have a prime-power basic $\FPdim$, as shown in \cite[Corollary 6.16]{nik}. Consequently, the previously mentioned perfect types are excluded. It follows that:

\begin{theorem} \label{thm:625}
Every perfect odd-dimensional modular fusion category of rank less than $25$ is trivial, and so everyone of rank less than $625$ ($=25^2$) is simple.
\end{theorem}
\textbf{Proof of Theorem \ref{thm:MNSD25}} \begin{proof} By Theorem \ref{thm:625}, there remain to address the non-perfect types above. Their rank is always $17 + 2a < 25$, which implies $a < 4$. As outlined in \S \ref{sub:univ}, the modular grading results in a partition indexed by the pointed part, with each component having the same $\FPdim$, in particular the $\FPdim$ of the pointed part divides the global $\FPdim$.

\begin{itemize}
\item First, let's examine the type $[[1, 9], [3, 8], [9, 2a]]$. The $\FPdim$ for this type is $81(1 + 2a)$. Consequently, each partition component must have $\FPdim = 9(1 + 2a)$. If $a > 0$, a component with $9$ must have $\FPdim \ge 81$. This leads to $81 \le 9(1 + 2a)$, resulting in $a \ge 4$, contradiction. Thus, $a = 0$. The modular partition must be $$[[1,1,1,1,1,1,1,1,1],[3],[3],[3],[3],[3],[3],[3],[3]],$$ which contradicts Theorem \ref{thm:ModCrit} (1).
\item Regarding the second type $[[1, 7], [3, 2], [5, 8], [15, 2a]]$, the $\FPdim$ of the pointed part equaling $7$ is not a divisor of the global $\FPdim = 225(1 + 2a)$, for $0 \le a < 4$, except $a=3$. Therefore, each partition component must have $\FPdim = 225(1 + 2 \times 3)/7 = 225 = 15^2$, so the modular partition must be $$[[1,1,1,1,1,1,1,3,3,5,5,5,5,5,5,5,5],[15],[15],[15],[15],[15],[15]],$$ which contradicts Theorem \ref{thm:ModCrit} (1).

\item Lastly, consider the third type $[[1, 3], [3, 8], [5, 6], [15, 2a]]$. If $a=1$ then the modular partition must be $$[[1,1,1,3,3,3,3,3,3,3,3,5,5,5,5,5,5],[15],[15]],$$ which contradicts Theorem \ref{thm:ModCrit} (1). If $a=2$, then there is no modular partition. If $a=3$, then the global $\FPdim = 3^25^2 7$, but its powerless prime factor $p=7$ does not divide $\FPdim(\mathcal{C}_{pt}) = 3$, contradicting \cite[Theorem 1.17]{BuPa23}. Alternatively, see Remark \ref{rk:alter}. Hence, $a = 0$. Following this, applying the fusion ring solver outlined in \S \ref{sub:full} to the type $[[1, 3], [3, 8], [5, 6]]$ yields two fusion rings. Applying \texttt{STmatrix2} to these fusion rings provides $3$ modular data, detailed in \cite[\S\ref{sec:OddMD}]{ABPPsupp}. \qedhere
\end{itemize}
\end{proof}

\begin{remark} \label{rk:alter}
Here is an alternative proof for the case $a=3$ in the last item above. The neutral component $\mathcal{C}_e$ must be of type $[[1, 3], [3, 8], [15, 2]]$ and $\FPdim = 3^1 5^2 7$, so its modularization $\mathcal{M}$ would have $\FPdim = 5^2 7$, hence cannot have basic $\FPdim = 3$, thus must be of type $[[1,25],[5,6]]$, but each component of the modular partition (for $\mathcal{M}$) would have $\FPdim = 7$, and so $5^2 \le 7$, contradiction.
\end{remark}  

\begin{definition} \label{def:anofree}
A modular data is called \emph{anomaly-free} if its Gauss sums are equal ($p_+ = p_-$), see Definition \ref{def:MD}.
\end{definition}

\begin{lemma} \label{lem:anofree}
A modular data is anomaly-free if and only if $p_+= \pm \sqrt{\dim}$, if and only if the central charge $c \in \{0,4\}$. 
\end{lemma}
\begin{proof}
Recall from Definition \ref{def:MD} that $p_\pm := \sum_{i=1}^{r} d_i^2 (\theta_i)^{\pm 1}$. Thus $p_+$ and $p_-$ are complex-conjugate. So anomaly-free is equivalent to $p_+$ real. Now, $p_+=\sqrt{\dim}\zeta_8^c$, thus it is real if and only if $\zeta_8^c = \pm 1$, if and only if $c=0$ or $4$.
\end{proof}

\begin{remark} \label{rk:gaps}
As highlighted in Remark \ref{rk:IntroGaps}, gaps have been identified in the literature:
\begin{enumerate}
\item In \cite[Theorem 4.2, proof of Case (viii) $\FPdim(\mathcal{C}_{pt}) = p$]{BGHKNNPR}, on page 727, the assertion that the anomaly-freeness (as defined in their reference \cite{DLN}, see Definition \ref{def:anofree}) necessarily leads to $p_+ = pq$ is incorrect, see Lemma \ref{lem:anofree}, it may also be $-pq$, as for the MD described in \cite[\S\ref{sec:OddMD}]{ABPPsupp}, thus also allowing $p | (q+1)$.
%which align with this case, satisfy the equation on top of page 727, and are anomaly-free, yet exhibit a Gauss sum $p_+ = -pq$. Similarly, for the $8$ MD of rank $8$ mentioned in \S \ref{NPMD} (all categorifiable and anomaly-free), the 6 ones with a central charge $c=0$ have $p_+ = pq$, and the 2 ones with $c=4$ have $p_+ = -pq$, as also verified in \cite{NRW23}. In general, see Lemma \ref{lem:anofree}.
\item In \cite[Theorem 6.3 (b), proof of Case $|\mathcal{G}(\mathcal{C})|=3$]{CzPl}, on page 1936, the deduction ``Hence $l \le 24$'' in the seventh last line is accurate, except when $c_{X_1} = 1$, which permits $l=5$, thereby accommodating the type $[[1, 3], [3, 8], [5, 6]]$. 
\end{enumerate}
\end{remark}
Following our paper, \cite{CzPl} was corrected on arXiv, and \cite{GPR24} introduces modular categorifications for these new MD.
\begin{remark} \label{rk:WithSeb}
Regarding the modular categorifications $\mathcal{C}$ of these new MD in \cite[\S\ref{sec:OddMD}]{ABPPsupp}, a discussion with Sebastian Burciu revealed that the Grothendieck ring of the braided adjoint fusion subcategory $\mathcal{D} = \mathcal{C}_{ad}$, which is of rank $11$, $\FPdim$ $75$, type $[[1,3],[3,8]]$, and basis $\{b_g\}_{g \in C_3} \cup \{x_i,x_i^*\}_{i \in \{1,2,3,4\}}$, is equal to $\ch(G)$, where $G = C_5^2\rtimes C_3$ is the unique non-Abelian finite group of order $75$. In fact, an application of \S \ref{sub:full} shows a unique MNSD cyclotomic fusion ring of this type.  Furthermore, $\mathcal{D}$ is not symmetric, as indicated by the $S$-matrices mentioned in \cite[\S\ref{sec:OddMD}]{ABPPsupp}. In fact, the M\"uger center $\mathcal{Z}_2(\mathcal{D})$ is $\mathcal{C}_{pt}$, pointed of rank $3$. Therefore, $\mathcal{D}$ can be $\Rep(G)$, albeit with an unusual braiding (see \cite{Davy}), or, more broadly, a Jordan-Larson category \cite{JoLa} with an $\FPdim$ of $3 \times 5^2$. Finally, according to the $S$-matrices again, $\mathcal{C}$ is a minimal modular extension of $\mathcal{D}$, see \cite{Mu03} and \cite[\S 1.1]{JFR24}.
\end{remark}

There are non-pointed and non-perfect odd-dimensional modular fusion categories of rank $25$, exemplified by $\mathcal{Z}(\VVec_{C_7 \rtimes C_3}^{\omega})$. Furthermore, \cite{CzPl2} demonstrates that, up to equivalence, no additional examples exist. Consequently, our attention must now turn to the examination of the perfect case.
\begin{proposition} \label{prop:MNSDper25}
A perfect odd-dimensional modular fusion category of rank $25$, if any, must have one of the following $3$ types:
\begin{enumerate}[label=\arabic*.,leftmargin=*]
%\item $[[1,1],[35,4],[39,2],[65,4],[91,4],[273,2],[455,8]]$,
%\item $[[1,1],[39,4],[65,2],[189,2],[315,2],[585,2],[1365,2],[2457,2],[4095,8]]$,
\item $[[1,1],[75,2],[91,4],[175,2],[585,2],[975,2],[2275,2],[4095,2],[6825,8]]$,
\item $[[1,1],[75,2],[91,4],[175,2],[975,2],[2275,2],[2925,4],[6825,8]]$,
\item $[[1,1],[135,4],[165,2],[189,2],[315,2],[385,2],[1155,2],[2079,2],[3465,8]]$.
\end{enumerate}
\end{proposition}
\begin{proof}
We are reduced to the simple case by Theorem \ref{thm:625}. Moreover, $q=1$; otherwise, the type would be a completion (via the process of adding 1s to the Egyptian fraction with squared denominators) of a perfect type with a rank less than $25$. However, as previously mentioned, all such types possess some entries that are prime-power, so they do not conform to the simple case as outlined in \cite[Corollary 6.16]{nik}. The investigation is documented in \texttt{InvestMNSDPerfectR25.txt}. We started with $91$ possible types using the aforementioned method combined with \cite[Corollary 6.16]{nik}. Among them, only $15$ types pass \verb*|Theorem8_7Check| (i.e., Theorem \ref{thm:StrongPrime}), and then $7$ pass \texttt{TypesCriteria}. Using our fusion ring solver in \S \ref{sub:DimPar} and HPC, we reduced the list to the $3$ mentioned types.
\end{proof}

These three types are resolved in a subsequent paper; see Remark~\ref{ENOcrit}.

\begin{appendices}
\section{Proof by Geoff Robinson} \label{sec:geoff}

This proof, originally due to Geoff Robinson \cite{PalMO2} and slightly simplified here, completes the argument for Theorem \ref{thm:r/3} by establishing that $ r_G \geq 3p - 1 $ when $ G $ is non-Abelian. Most of the steps are straightforward, though the final part relies on a deeper result from modular representation theory, by R. Brauer \cite{Brauer}.

\begin{enumerate}
    \item\label{step:largest_prime} We may assume that $p$ is the largest prime dividing $|G|$; if not, we replace $p$ with the largest one.
    
    \item\label{step:conjugacy_bound} The number of conjugacy classes of the centralizer $C_G(a)$ is at least the order of its center $Z(C_G(a))$. Furthermore, since $Z(C_G(a))$ contains both $Z(G)$ and $\langle a \rangle$, we have:
    \[
    |Z(C_G(a))| \geq \operatorname{lcm}(\operatorname{ord}(a), |Z(G)|).
    \]
From this point onward, assume that $ r_G \leq 3p - 2 $. Under this condition, there can be at most two non-conjugate elements of order $ p $, and no element $ g $ of order $ kp $ for any integer $ k > 1 $. The existence of such an element would imply $$ r_G \geq \operatorname{ord}(g) + \operatorname{ord}(g^k) = kp + p  = (k+1) p \ge 3p > 3p - 2,$$ since two elements with different order, like $g$ and $g^k$, cannot be conjugate. This contradicts the assumption.

    \item\label{step:centralizer_pgroup}  Let $x \in G$ be an element of order $p$. Suppose $C_G(x)$ is not a $p$-group. Then it contains an element $y$ of prime order $q \neq p$. Since $x$ and $y$ commute and their orders are coprime, the element $xy$ has order $pq$, contradicting the bound in Step~\ref{step:conjugacy_bound} since $q>1$. Therefore, $C_G(x)$ must be a $p$-group.

    \item \label{step:lemma} 
%\begin{lemma} 
        \textbf{Lemma}: Let $H$ be a non-cyclic $p$-group. Then $H$ contains at least $3p - 2$ conjugacy classes.
%\end{lemma}

\begin{proof}
        If $H$ is abelian, then it must have order $p^n$ with $n \ge 2$ (since it is non-cyclic), and the number of conjugacy classes equals $|H| \geq p^2 \geq 3p - 2$.

        Recall that a group of order $p^2$ must be Abelian. Suppose $H$ is non-abelian. Then $|H| = p^n$ with $n \geq 3$. Let $g_1, \dots, g_m$ be representatives of the distinct conjugacy classes in $H$. By the class equation,
        \[
        |H| = \sum_{i=1}^m |H : C_H(g_i)|.
        \]
        One term in the sum is $1$, corresponding to the neutral element. Since all the indices $|H : C_H(g_i)|$ are powers of $p$, and the sum is congruent to $1 \bmod p$, there are at least $p$ terms equal to $1$.

        The remaining $m - p$ terms sum to $|H| - p$. Modulo $p^2$, we conclude that at least $p - 1$ of these terms must be at most $p$. Therefore, we have at least $2p - 1$ terms each at most $p$, and their total contribution is at most:
        \[
        p + (p - 1)p = p^2.
        \]
        Since $|H| = p^n$ with $n \geq 3$, modulo $p^3$, we get that at least $p - 1$ additional terms must be at most $p^2$. Hence, altogether, there are at least $3p - 2$ conjugacy classes.
    \end{proof}

    \item\label{step:cyclic_centralizer} If $C_G(x)$ is not cyclic, then by Step~\ref{step:lemma}, it has at least $3p - 2$ conjugacy classes. Since $C_G(e) = G$ has at least $2$ conjugacy classes, we obtain $r_G \geq 3p$, which contradicts the assumption. Hence, $C_G(x) = \langle x \rangle$ for every $x \in G$ of order $p$.

    \item\label{step:sylow_order} Let $H$ be a Sylow $p$-subgroup of $G$. If $|H| = p^n$ with $n>1$, by the class equation modulo $p$, $|Z(H)| = p^m$ with $m \ge 1$, so $H$ has a central element $x$ of order $p$, and $C_G(x)$ contains $H$, contradicting Step~\ref{step:cyclic_centralizer}. % Thus, the Sylow $p$-subgroups of $G$ are cyclic of order $p$.

    \item\label{step:odd_prime} By Step~\ref{step:sylow_order}, $G$ has a cyclic Sylow $p$-subgroup of order $p$. If $p = 2$, then by Step~\ref{step:largest_prime}, $G \cong C_2$, but $G$ is non-Abelian. Therefore, $p$ is odd, and the Sylow $ p $-subgroup of $ G $ is  $H = \langle x \rangle = C_G(x) $, of order $p$. 
    %\textcolor{blue}{Only point where p max required. Really useful?}
   
\item \label{step:normal_case}
Let $ G_1 = \{ g \in G : gH = Hg \} $ be the normalizer of $ H $ in $ G $. Now, $ H $ is a normal (and hence unique) Sylow $ p $-subgroup of $ G_1 $, as all Sylow subgroups are conjugate. Let $ |G_1| = ep $, with $e$ and $p$ coprime. Brauer’s result \cite{Brauer} establishes that the entire group $ G $ — not merely $ G_1 $ — has at least $ e + \frac{p - 1}{e} $ conjugacy classes. Moreover, two distinct elements $ x_1, x_2 \in H $ are conjugate in $ G $ if and only if they are conjugate in $ G_1 $. This is because any $ g \in G $ satisfying $ gx_1g^{-1} = x_2 $ must, by definition, lie in $ G_1 $.

There is an action of $ G_1 $ on $ H $ by conjugation, yielding a homomorphism  
\[ \Phi : G_1 \to \Aut(H) \cong C_{p - 1}. \]  
The kernel of this homomorphism is $ H $, since $ H = C_G(x) $, making $ H $ self-centralizing.

Each automorphism $ \theta \in \Aut(H) $ is determined uniquely by the image of $ x $, so the number of conjugates of $ x $ in $H$ is exactly $ |\Phi(G_1)| $, and the same holds for any generator of $ H $. Thus, the index of $ \Phi(G_1) $ in $ \Aut(H) $ is 1 if all order $p$ elements in $ G $ are conjugate; otherwise, it is $2$. This follows from Step \ref{step:conjugacy_bound}, which shows that there can be at most two conjugacy classes of such elements—and these classes must be of equal size, since the centralizer of any generator of $ H $ is always $ H $ itself. %These are the only two possibilities. 
Since  
\[
|G_1| = ep = |\Phi(G_1)| \cdot |\ker \Phi| = |\Phi(G_1)| \cdot p,
\]  
it follows that $ |\Phi(G_1)| = e $. Therefore,  
$
\frac{p - 1}{e} \in \{1, 2\}
$  
is the number of conjugacy classes of elements of order $ p $ in $ G $.
As a subgroup of a cyclic group, $ \Phi(G_1) $ is also cyclic, so generated by a element of order $e$, and $G$ contains a  element of order divisible by $e$. 

If $ e > 1 $, then any element whose order is divisible by $ e $ must be nontrivial. Moreover, such an element cannot have order $ p $, since $ e $ and $ p $ are coprime.
Putting everything together, the number of conjugacy classes satisfies:
\[
r_G \geq \left(e + \frac{p - 1}{e}\right) + e + \frac{p-1}{e}p \geq 3p - 1,
\]
\begin{itemize}
\item the first term bounds the number of conjugacy classes in $C_G(1) = G$ by Brauer's result,
\item the second corresponds to the non-trivial element of order divisible by $e$,
\item the third accounts for $ \frac{p-1}{e} $ conjugacy classes of elements of order $ p $.
\end{itemize}    
Observe that the second inequality follows immediately upon setting $\frac{p - 1}{e} = 1$ or $2$.

Finally, if $ e = 1 $, then $ p = 3 $. Given that $ p $ is the largest prime divisor of $ |G| $, that $ 9 $ does not divide $ |G| $, and that $ G $ is non-Abelian, it follows that $ G $ contains an element of order 2. This element can thus be used in our bound in place of one whose order is divisible by $ e $.
\end{enumerate}

%\end{landscape}
%\addtocontents{toc}{\protect\setcounter{tocdepth}{2}}
\end{appendices}

\section*{About the Supplementary Material}

The complete fusion and modular data are provided separately in the Supplementary Material accompanying this paper, which is cited as~\cite{ABPPsupp}.

\section*{Acknowledgments}
We sincerely thank the anonymous referees for their careful review, which greatly improved this paper. We extend our gratitude to Sebastian Burciu, César Galindo, Eric Rowell, Yilong Wang and Andrew Schopieray for their invigorating interest in this work and for the insightful discussions that have greatly contributed to it. Thanks also to Dave Benson and Geoff Robinson for exchanges regarding finite group theory. The third author's research was supported by the BIMSA Start-up Research Fund, and the Foreign Youth Talent Program sponsored by the Ministry of Science and Technology of China (Grant No. QN2021001001L), and the National Natural Science Foundation of China (NSFC, Grant no. 12471031). The high-performance computing cluster at the University of Osnabrück, which was instrumental in facilitating our computations, was funded by the DFG grant 456666331. We also thank Lars Knipschild, the system administrator of the HPC facilities, for his support and technical assistance.

\vspace*{.5cm}

\noindent \textbf{Availability of data and materials.} Data for the computations in this paper are available on reasonable request from the authors. The softwares used for the computations can be downloaded from the URLs listed in the references.

\vspace*{.15cm}

\noindent \textbf{Conflict of interest statement.} On behalf of all authors, the corresponding author declares that there are no conflicts of interest.


\begin{thebibliography}{99}

\bibitem{A348625}
{\sc M.A.~Alekseyev}, {\em Number of Egyptian fractions with squared denominators}, Entry A348625 in OEIS, \url{http://oeis.org/A348625}.

\bibitem{MaxScripts}
{\sc M.A.~Alekseyev}, {\em Egyptian Fractions}, GitHub repository, \url{https://github.com/maxale/egyptian-fractions}.

\bibitem{ABPPsupp}
{\sc M.A.~Alekseyev, W.~Bruns, S.~Palcoux, F.V. Petrov}, {\em Supplementary Material, Classification of integral modular data up to rank 13},
\url{https://github.com/sebastienpalcoux/Fusion-Categories/blob/main/ModularPaper/SupplementaryMaterial.pdf}

\bibitem{Brauer} 
{\sc R.~Brauer}, \emph{Investigations on group characters}. Ann Math. 42 (1941), 936--958

\bibitem{BGNPRW}
{\sc P.~Bruillard, C.~Galindo, S-H.~Ng, J.Y.~Plavnik, E.C.~Rowell, Z.~Wang}, {\em On the classification of weakly integral modular categories}. J. Pure Appl. Algebra 220 (2016), no. 6, 2364--2388. 

\bibitem{BGHKNNPR}
{\sc P.~Bruillard, C.~Galindo, S.M.~Hong, Y.~Kashina, D.~Naidu, S.~Natale, J.Y.~Plavnik, E.C.~Rowell}, {\em Classification of integral modular categories of Frobenius-Perron dimension $pq^4$ and $p^2q^2$}. Canad. Math. Bull. 57 (2014), no. 4, 721--734.

\bibitem{BNRW}
{\sc P.~Bruillard, S.-H.~Ng, E.C.~Rowell, Z.~Wang}, {\em Rank-finiteness for modular categories}, J. Am. Math. Soc. 29(3), 857–881 (2016).  

\bibitem{BrRo}
{\sc P.~Bruillard, E.C.~Rowell}, {\em Modular categories, integrality and Egyptian fractions.} Proc. Amer. Math. Soc. 140 (2012), no. 4, 1141--1150.

%\bibitem{Bru00}
%{\sc A.~Bruguières}, {\em Catégories prémodulaires, modularisations et invariants des variétés de dimension 3.} (French), [[Premodular categories, modularizations and invariants of 3-manifolds]] Math. Ann. 316 (2000), no. 2, 215--236. 

\bibitem{BrGu}
{\sc W.~Bruns and J.~Gubeladze.}, {\em Polytopes, rings, and $K$-theory.} Springer Monographs in Mathematics, Springer, Dordrecht, 2009.

\bibitem{BP25} 
{\sc W.~Bruns, S.~Palcoux}, {\em Classifying simple integral fusion rings}, work in progress (2025)

%\bibitem{Bur1904}
%{\sc W.~Burnside}, {\em On Groups of Order $p^{\alpha}q^{\beta}$}. Proc. London Math. Soc. (2) 1 (1904), 388--392.

\bibitem{Norma} 
{\sc W.~Bruns, B.~Ichim, C.~S\"oger and U.~von der Ohe}, {\em Normaliz. Algorithms for rational cones and affine monoids}. Available at \url{https://www.normaliz.uni-osnabrueck.de}.

\bibitem{NorManual} 
{\sc W.~Bruns, M.~Horn}, {\em Manual - Normaliz 3.11.1}. Available at \url{https://github.com/Normaliz/Normaliz/blob/master/doc/Normaliz.pdf}.


%\bibitem{Cra}
%{\sc D.A.~Craven}, {\em Perfect groups whose character degrees square divide its order}, Mathematics Stack Exchange, \url{https://math.stackexchange.com/q/4574920} (version: 2022-11-13)

\bibitem{BuPa23}
{\sc S.~Burciu, S.~Palcoux}, {\em Burnside type results for fusion categories}, Commun. Contemp. Math. (2025), Paper No. 2650007.

\bibitem{CGR}
{\sc A. Coste, T. Gannon, P. Ruelle}, {\em Finite group modular data}, Nuclear Phys. B 581 (2000), no. 3, 679–717.   

\bibitem{CzPl}
{\sc A.~Czenky, J.~Plavnik}, {\em On odd-dimensional modular tensor categories}, Algebra Number Theory 16 (2022), no. 8, 1919--1939. Corrected version arXiv:2007.01477 (2024).

\bibitem{CzPl2}
{\sc A.~Czenky, W.~Gvozdjak, J.~Plavnik}, {\em Classification of low-rank odd-dimensional modular categories}, J. Algebra 655 (2024), 223–293.

\bibitem{DHW13}
{\sc O.~Davidovich, T.~Hagge, Z.~Wang}, {\em On Arithmetic Modular Categories}, arXiv:1305.2229.

\bibitem{Davy}
{\sc A.~Davydov}, {\em Quasitriangular structures on cocommutative Hopf algebras}, Preprint, arXiv:9706007.

\bibitem{DeGaPlRoZh}
{\sc C.~Delaney, C.~Galindo, J.~Plavnik, E.C.~Rowell, Q.~Zhang}, {\em Braided zesting and its applications}. Comm. Math. Phys. 386 (2021), no. 1, 1--55.

\bibitem{DP25}
{\sc J.~Dong, S.~Palcoux}, {\em A new criterion for integral modular categorification}, arXiv:2510.10311.

\bibitem{DLN}
{\sc C.~Dong, X.~Lin, S-H.~Ng}, {\em Congruence property in conformal field theory}. Algebra Number Theory 9 (2015), no. 9, 2121--2166.
 
%\bibitem{DoNaVe}
%{\sc J.~Dong, S.~Natale, L.~Vendramin}, {\em Frobenius property for fusion categories of small integral dimension}, J. Algebra Appl. 14 (2015), no. 2, 1550011, 17 pp.

%\bibitem{EGGS}
%{\sc P.~Etingof, S.~Gelaki, R.~Guralnick, J.~Saxl}, {\em Biperfect Hopf Algebras}, J. Algebra, 232, 331-335 (2000).

 \bibitem{EGNO}
{\sc P.~Etingof, S.~Gelaki, D.~Nikshych, and V.~Ostrik}, {\em Tensor Categories}, American Mathematical Society, (2015).
\newblock Mathematical Surveys and Monographs Volume 205.

\bibitem{ENO05}
{\sc P.~Etingof, D.~Nikshych, and V.~Ostrik}, {\em On fusion categories}. Ann. of Math. (2) 162 (2005), no. 2, 581--642.

\bibitem{ENO11}
{\sc P.~Etingof, D.~Nikshych, and V.~Ostrik}, {\em Weakly group-theoretical and solvable fusion categories}, Adv. Math., 226 (2011), pp.~176--205. 

\bibitem{ENO21}
{\sc P.~Etingof, D.~Nikshych, and V.~Ostrik}, {\em On a necessary condition for unitary categorification of fusion rings}, Quantum groups, Hopf algebras, and applications—in memory of Earl Jay Taft, 45--53, Contemp. Math., 814, (2025).

\bibitem{GPR24}
{\sc C.~Galindo, J.~Plavnik, and E.~Rowell}, {\em Integral non-group-theoretical modular categories of dimension $p^2q^2$},	Bull. Belg. Math. Soc. Simon Stevin 31 (2024), no. 4, 516--526.

\bibitem{GMR24}
{\sc C.~Galindo, G.~Mora, and E.~Rowell}, {\em Braided zestings of Verlinde modular categories and their modular data.} Comm. Math. Phys. 405 (2024), no. 10, Paper No. 249, 34 pp.

%\bibitem{Gan02}
%{\sc T.~Gannon}, {\em The automorphisms of affine fusion rings}. Adv. Math. 165 (2002), no. 2, 165--193.

\bibitem{GeNaNi}
{\sc S.~Gelaki, D.~ Naidu, D.~ Nikshych}, {\em Centers of graded fusion categories}. Algebra Number Theory 3 (2009), no. 8, 959--990.

\bibitem{GrMo}
{\sc A.~Gruen, S.~Morrison}, {\em Computing modular data for pointed fusion categories.} Indiana Univ. Math. J. 70 (2021), no. 2, 561--593, \url{https://tqft.net/web/research/students/AngusGruen/Modular_Data/6/1/}

%\bibitem{NNW}
%{\sc D.~Naidu, D.~Nikshych, S.~Witherspoon}, {\em Fusion subcategories of representation categories of twisted quantum doubles of finite groups}. Int. Math. Res. Not. 2009, no. 22, 4183–4219.

\bibitem{JFR24}
{\sc T.~Johnson-Freyd, D.~Reutter}, {\em Minimal nondegenerate extensions}. J. Amer. Math. Soc. 37 (2024), no. 1, 81--150. 

\bibitem{JoLa}
{\sc D.~Jordan, E.~Larson, Eric},  {\em On the classification of certain fusion categories}. J. Noncommut. Geom. 3 (2009), no. 3, 481--499. 

\bibitem{kac90}
{\sc V.G.~Kac}, {\em Infinite-dimensional Lie algebras}. Third edition. Cambridge University Press, Cambridge, 1990. {\rm xxii}+400 pp.

\bibitem{LPRinter}
{\sc Z.~Liu, S.~Palcoux and Y.~Ren}, {\em Interpolated family of non-group-like simple integral fusion rings of Lie type}, Internat. J. Math. 34 (2023), no. 6, Paper No. 2350030, 51 pp., DOI: 10.1142/S0129167X23500301

\bibitem{Ma95}
{\sc G.~Mason}, {\em The quantum double of a finite group and its role in conformal field theory}. Groups '93 Galway/St. Andrews, Vol. 2, 405--417, London Math. Soc. Lecture Note Ser., 212, Cambridge Univ. Press, Cambridge, 1995

\bibitem{Mu03}
{\sc M.~Müger}, {\em On the structure of modular categories}. Proc. London Math. Soc. (3) 87 (2003), no. 2, 291--308.

\bibitem{NN}
{\sc D. Naidu, D. Nikshych}, {\em Lagrangian subcategories and braided tensor equivalences of twisted quantum doubles of finite groups}. Comm. Math. Phys. 279 (2008), no. 3, 845–872.      

\bibitem{NRWW}
{\sc S.H. Ng, E.C. Rowell, Z. Wang, X.-G. Wen}, {\em Reconstruction of modular data from ${\rm SL}_2(\Bbb Z)$ representations.} Comm. Math. Phys. 402 (2023), no. 3, 2465--2545.  

\bibitem{NRW23}
{\sc S.-H.~Ng, E.C.~Rowell, X.-G.~Wen}, {\em Classification of modular data up to rank 12}, arXiv:2308.09670. 

%\bibitem{NS07} 
%{\sc S.-H.~Ng, P.~Schauenburg}, {\em Frobenius-Schur indicators and exponents of spherical categories.} Adv. Math. 211 (2007), no. 1, 34--71. 

\bibitem{NWW}
{\sc S.H. Ng, Y. Wang, S. Wilson}, {\em On symmetric representations of $\operatorname{SL}_2(\mathbb{Z})$}. Proc. Amer. Math. Soc. 151 (2023), no. 4, 1415–1431.    

\bibitem{NWW2}
{\sc S.H. Ng, Y. Wang, S. Wilson}, {\em SL2Reps, Constructing symmetric representations of SL(2,Z)}, Version 1.0, Dec 2021. GAP package.    

\bibitem{NWZ}
{\sc S.H. Ng, Y. Wang,  Q. Zhang}, {\em Modular categories with transitive Galois actions}. Comm. Math. Phys. 390 (2022), no. 3, 1271--1310.    

\bibitem{NiRi}
{\sc W.D.~Nichols, M.B.~Richmond}, {\em The Grothendieck group of a Hopf algebra.} J. Pure Appl. Algebra 106 (1996), no. 3, 297--306

\bibitem{nik}  
{\sc D.~Nikshych}, \emph{Morita equivalence methods in classification of fusion categories}, Hopf algebras and tensor categories, 289–325, Contemp. Math., 585, Amer. Math. Soc. (2013).

\bibitem{NW}
{\sc A. Nobs, J. Wolfart},  {\em Die irreduziblen Darstellungen der Gruppen $\operatorname{SL}_{2}(\Bbb Z_{p})$, insbesondere $\operatorname{SL}_{2}(\Bbb Z_{2})$. II]}. (German) Comment. Math. Helv. 51 (1976), no. 4, 491–526.   

\bibitem{OstR3}
{\sc V.~Ostrik}, \emph{Pivotal fusion categories of rank 3.} Mosc. Math. J. 15 (2015), no. 2, 373--396, 405.

\bibitem{PalMO}
{\sc S.~Palcoux}, \emph{Number of conjugacy classes of pairs of commuting elements}, MathOverflow, 2024-03-13, \url{https://mathoverflow.net/q/466800} 

\bibitem{PalMO2}
{\sc S.~Palcoux}, \emph{Number of conjugacy classes of pairs of commuting elements II}, MathOverflow, 2024-04-04, \url{https://mathoverflow.net/q/468354}

\bibitem{A371059}
{\sc S.~Palcoux}, \emph{Number of conjugacy classes of pairs of commuting elements in the alternating group $A_n$}, Entry A371059 in OEIS, \url{https://oeis.org/A371059}

\bibitem{data}
{\sc S.~Palcoux}, {\em Fusion Categories}, GitHub repository, \url{https://github.com/sebastienpalcoux/Fusion-Categories/tree/main/ModularPaper}.

\bibitem{PSYZ}
{\sc J. Plavnik, A. Schopieray, Z. Yu, Q. Zhang}, {\em Modular tensor categories, subcategories, and Galois orbits}, Transformation Groups 29 (2024). https://doi.org/10.1007/s00031-022-09787-9

%\bibitem{PalMXE}
%{\sc S.~Palcoux}, {\em Perfect groups whose character degrees square divide its order and are not prime-power}, Mathematics Stack Exchange, \url{https://math.stackexchange.com/q/4632044/84284} (version: 2023-02-04)

\bibitem{sage}
The Sage Developers, \emph{SageMath, the Sage Mathematics Software System (Version 10.3)}, 2024. Available at \url{https://www.sagemath.org}.

\bibitem{A002966}
{\sc N.J.A.~Sloane, R.G.~Wilson v}, {\em Egyptian fractions}, Entry A002966 in OEIS, \url{http://oeis.org/A002966}

\end{thebibliography}
\end{document}